\documentclass[preprint,12pt]{elsarticle}
\usepackage[top=2.5cm, bottom=2.5cm, left=3cm, right=3cm]{geometry}   
\usepackage{amsmath,amsthm,amsfonts,amssymb,mathrsfs}
\usepackage{mathtools}
\usepackage{hyperref}

\usepackage{makecell}
\usepackage{tabularx}
\usepackage{xcolor}
\usepackage{seqsplit}

\newtheorem*{theorem*}{Theorem}
\newtheorem{theorem}{Theorem}[section]
\newtheorem{definition}[theorem]{Definition}
\newtheorem{example}[theorem]{Example}

\newtheorem{remark}[theorem]{Remark}
\newtheorem{lemma}[theorem]{Lemma}
\newtheorem{corollary}[theorem]{Corollary}
\newtheorem{proposition}[theorem]{Proposition}
\journal{Finite Fields and Their Applications}

\DeclarePairedDelimiter\ceil{\lceil}{\rceil}
\DeclarePairedDelimiter\floor{\lfloor}{\rfloor}

\def\la{\lambda}

\def\P{\mathbf P}
\def\N{\mathbb N}
\def\Z{\mathbb Z}
\def\Div{{\rm Div}}
\def\deg{{\rm deg}\,}
\def\lub{{\rm lub}}
\def\n{\mathbf n}

\def\v{\mathbf v}

\def\j{\mathbf j}
\def\a{\mathbf a}

\begin{document}
\begin{frontmatter}



\title{Weierstrass semigroups at totally ramified places of degree one on generalized Kummer extensions}


\author[1]{Huachao Zhang}
\ead{zhanghch56@mail2.sysu.edu.cn}

\author[1,2]{Chang-An Zhao\corref{cor1}}
\ead{zhaochan3@mail.sysu.edu.cn}
\cortext[cor1]{Corresponding author.}
\affiliation[1]{organization={School of Mathematics, Sun Yat-sen University},%Department and Organization
            city={Guangzhou},
            postcode={510275}, 
            state={Guangdong},
            country={China}}
\affiliation[2]{organization={Guangdong Provincial Key Laboratory of Information Security Technology},%Department and Organization
            city={Guangzhou},
            postcode={510006}, 
            state={Guangdong},
            country={China}}           

%% Abstract
\begin{abstract}
    We explicitly describe the set of gaps and the Weierstrass semigroup at a totally ramified place of degree one on a generalized Kummer extension defined by the affine equation $y^m = f(x)$ over $K$, an algebraic extension of $\mathbb{F}_q$, where $f(x)\in K(x)$ and $K$ is not required to contain a primitive $m$-th root of unity.
    For the different totally ramified places of degree one, our description takes a unified form.
    Then we provide a necessary and sufficient condition for the Weierstrass semigroup at a totally ramified place of degree one to be symmetric.
	Furthermore, we investigate the minimal generating set of the Weierstrass semigroups at many totally ramified places of degree one.
	We not only explicitly describe the minimal generating set, but also provide functions whose coefficients of pole divisors lie in the minimal generating set.
    Finally, we apply our results to generalized Kummer extensions with the same multiplicities, as well as function fields of GGS curves and subcovers of the BM curve.
\end{abstract}

%% Keywords

\end{frontmatter}

\section{Introduction}
	Let $K$ be an algebraic extension of $\mathbb{F}_q$, let $F/K$ be a function field with full constant field $K$, and let $P$ be a rational place of $F$. 
    The Weierstrass semigroup at $P$, denoted by $H(P)$, is a classical object in the theory of algebraic function fields and algebraic curves; for example, see \cite{castellanosWeierstrassSemigroupsPure2024,mendozaKummerExtensions2023,abdonWeierstrassPoints2019,beelenWeierstrassSemigroups2021,bartoliWeierstrassSemigroups2021a}.
    The theory of the Weierstrass semigroup at a rational place has numerous applications. For instance, it can be used to construct algebraic geometry codes with good parameters (see \cite{castellanosOneTwoPointCodes2016,castellanosWeierstrassSemigroupsPure2024,montanucciAGCodes2020});
    determine the automorphism group of an algebraic curve or a function field (see \cite{maAutomorphismGroups2016a,montanucciAutomorphismGroup2024,beelenWeierstrassSemigroups2023,beelenWeierstrassSemigroups2026a,beelenWeierstrassSemigroups2026});
    classify maximal function fields (see \cite{niemannNonisomorphicMaximal2025,beelenFamilyNonisomorphic2025}).

    The complementary set $G(P) := \N_0\setminus H(P)$ is called the set of gaps at $P$. If $F$ has genus $g\geq 1$, then $G(P)$ contains exactly $g$ elements.
    The largest element in $G(P)$, denoted by $F_{H(P)}$, is called the Frobenius number of $H(P)$. 
    We say that the Weierstrass semigroup $H(P)$ is symmetric if $F_{H(P)} = 2g-1$.
    Several results exist concerning the symmetry of Weierstrass semigroups. For the unique infinite place $Q_\infty$, Mendoza discussed sufficient conditions for $H(Q_\infty)$ to be symmetric in a Kummer extension \cite{mendozaKummerExtensions2023}; 
    G{\"u}neri, {\"O}zdemiry and Stichtenoth showed that $H(Q_\infty)$ is symmetric on the GGS curve \cite{guneriAutomorphismGroup2013}.
	For a totally ramified place $Q$ of degree one in a Kummer extension, Cotterill, Mendoza and Speziali characterized the symmetry of $H(Q)$ in certain cases \cite{cotterillGapSets2025}.

	Let $P_1,\dots,P_s$ be $s$ distinct rational places of the function field $F$. The Weierstrass semigroup $H(P_1,\dots,P_s)$ at several rational places has been widely studied in recent decades. In \cite{arbarelloGeometryAlgebraic1985}, Arbarello et al. initiated the theory of the Weierstrass semigroup at two rational places.
	This line of work was subsequently advanced by Kim \cite{kimIndexWeierstrass1994} and Homma \cite{hommaWeierstrassSemigroup1996}.
	The Weierstrass semigroup at several rational places was first considered by Carvalho and Torres \cite{carvalhoGoppaCodes2005}. In \cite{matthewsWeierstrassSemigroup2004}, Matthews introduced the notion of a minimal generating set for $H(P_1,\dots,P_s)$, denoted by $\tilde{\Gamma}(P_1,\dots,P_s)$, which is completely determined by certain special sets $\Gamma(P_1,\dots,P_k)$ for $1\leq k\leq s$.
	By computing $\Gamma(P_1,\dots,P_k)$ for $1\leq k\leq s$, Matthews determined the Weierstrass semigroup of any $s$-tuple of collinear rational places on a Hermitian curve.
	This result was then generalized to $s$-tuples of places on a quotient of the Hermitian curve in \cite{matthewsWeierstrassSemigroups2005}, and to those on Norm-trace curves in \cite{matthewsWeierstrassSemigroups2009,matthewsMinimalGenerating2010}.

	The Weierstrass semigroups at several rational places on other curves and function fields have also been well studied in the literature.
	For instance, Hu and Yang described the Weierstrass semigroups at several totally ramified places on specific Kummer extensions \cite{yangWeierstrassSemigroups2017,huMultipointCodesKummer2018}, 
	GGS curves \cite{huMultipointCodes2020}, and the third function field in a tower attaining the Drinfeld-Vl\u adu\c t bound \cite{yangWeierstrassSemigroups2024a}.
	In \cite{tizziottiWeierstrassSemigroup2018,castellanosWeierstrassSemigroup2018}, Tizziotti and Castellanos determined the Weierstrass semigroups at several rational places on GK curves and curves of the form $f(y) = g(x)$.
	Castellanos and Bras-Amor{\'o}s determined the Weierstrass semigroups at several rational places on the maximal curves that cannot be covered by the Hermitian curve in \cite{castellanosWeierstrassSemigroup2020}.
	Weierstrass semigroups at several places can be used to improve the minimum distance of an algebraic geometry code; see, for example, \cite{carvalhoGoppaCodes2005,duursmaImprovedTwoPoint2011,matthewsWeierstrassPairsMinimum2001,matthewsCodesSuzuki2004,sepulvedaWeierstrassSemigroup2014}.

	Let $f(x)\in K(x)$ and $m\geq 2$ with $\gcd(m,q)=1$. Consider the function field extension $F = K(x,y)/K(x)$ defined by the equation 
    \begin{equation}\label{kummer_extension}
    y^m = f(x) = \alpha\cdot\prod_{i=1}^rp_i(x)^{\lambda_i},
	\end{equation}
    where $\alpha\in K$, $\lambda_i\in \Z\setminus\{0\}$ for all $1\leq i \leq r$, and $p_1(x),\dots,p_r(x)\in K[x]$ are pairwise distinct monic irreducible polynomials.
	If $F/K(x)$ is a Kummer extension, then $K$ must contain a primitive $m$-th root of unity.
	The investigation of Weierstrass semigroups and sets of gaps on Kummer extensions has attracted attention in recent years.
    
	Assume that $\alpha=1$ and $\deg p_i(x) = 1$ for all $1\leq i\leq r$. If the infinite place is totally ramified in $F/K(x)$, we denote it by $Q_\infty$.
	Let $Q_1 \neq Q_\infty$ be another totally ramified place in $F/K(x)$.
	In Table \ref{table1}, the Weierstrass semigroup and the set of gaps at a totally ramified place have been explicitly described for some cases. Moreover, let $Q_1,\dots,Q_s$ be $s$ totally ramified places distinct from $Q_\infty$ in $F/K(x)$.
	In Table \ref{table2}, the sets $\Gamma(Q_1,\dots,Q_s)$ and $\Gamma(Q_\infty,Q_1,\dots,Q_s)$ have been explicitly described for some cases.
	However, there exist certain limitations in these previous works.
	All of the cases in Tables \ref{table1} and \ref{table2} require $\deg p_i(x) = 1$ for all $1\leq i\leq r$, i.e., $f(x)$ splits completely over $K$, and some cases require $1\leq \lambda_i\leq m-1$ for all $1\leq i\leq r$.
	Furthermore, all cases assume $f(x)\in K[x]$ and are treated in the framework of classical Kummer extensions; i.e., $K$ contains a primitive $m$-th root of unity.

    \begin{table}[htpb]
	\centering
	\caption{The Weierstrass semigroup and gap set at a totally ramified place from the literature}
	\vspace{6pt}
	\renewcommand{\arraystretch}{1.1}
	\begin{tabular}{|l|l|l|}
	\hline
	Reference & Case & Set\\ \hline
	\cite[Theorem 3.2]{castellanosOneTwoPointCodes2016} & $\lambda_1 = \dots = \lambda_r$ & $H(Q_\infty)$\\ \hline
	\cite[Theorem 3.4]{castellanosOneTwoPointCodes2016} & $\lambda_1 = \dots = \lambda_r$ and $m\equiv 1\pmod r$  &  $H(Q_1)$\\\hline
    \cite[Remark 2.8]{tafazolianCurveYn2019} & $1\leq \lambda_1\leq m-1$ and  $\lambda_2 = \dots = \lambda_r = 1$  & $H(Q_\infty)$ \\ \hline
	\cite[Theorem 3.2]{mendozaKummerExtensions2023} &  $1\leq \lambda_i\leq m-1$ for $1\leq i\leq r$  & $H(Q_\infty)$ \\ \hline
	\cite[Corollary 2]{huMultipointCodesKummer2018} &$\lambda_1=\dots = \lambda_r$ & $G(Q_1)$\\ \hline
	\cite[Proposition 4.3]{castellanosWeierstrassSemigroupsPure2024} & $1\leq \lambda_i\leq m-1$ for $1\leq i\leq r$ & $G(Q_\infty)$\\\hline
	\cite[Proposition 4.4]{castellanosWeierstrassSemigroupsPure2024} & $1\leq \lambda_i\leq m-1$ for $1\leq i\leq r$ & $G(Q_1)$\\
	\hline
	\end{tabular}
	\label{table1}
    \end{table}
	
	\begin{table}[htpb]
	\centering
	\caption{The sets $\Gamma(Q_1,\dots,Q_s)$ and $\Gamma(Q_\infty,Q_1,\dots,Q_s)$ from the literature}
	\vspace{6pt}
	\renewcommand{\arraystretch}{1.1}
	\begin{tabular}{|l|l|l|}
	\hline
	Reference & Case & Set\\ \hline
	\cite[Theorem 4.3]{castellanosOneTwoPointCodes2016} & $\lambda_1 = \dots = \lambda_r$ & $\Gamma(Q_\infty,Q_1)$\\ \hline
    \cite[Theorem 9]{yangWeierstrassSemigroups2017} & $\lambda_1 = \dots = \lambda_r$ and $s\geq 2$  &  $\Gamma(Q_1,\dots,Q_s)$ \\\hline
	\cite[Theorem 10]{yangWeierstrassSemigroups2017} & $\lambda_1 = \dots = \lambda_r$ and $s\geq 2$  &  $\Gamma(Q_\infty,Q_1,\dots,Q_s)$ \\\hline
	\cite[Proposition 4.6]{castellanosWeierstrassSemigroupsPure2024} & $1\leq \lambda_i\leq m-1$ for $1\leq i\leq r$ & $\Gamma(Q_1,Q_2)$\\ \hline
	\cite[Proposition 4.7]{castellanosWeierstrassSemigroupsPure2024} & $1\leq \lambda_i\leq m-1$ for $1\leq i\leq r$ & $\Gamma(Q_\infty,Q_1)$\\ \hline
	\cite[Corollary 3.5]{castellanosGeneralizedWeierstrass2026} & \makecell{ $f(x)\in K[x]$ and $s\geq 2$} & $\Gamma(Q_1,\dots,Q_s)$  \\ \hline
    \cite[Corollary 3.5]{castellanosGeneralizedWeierstrass2026} & \makecell{ $f(x)\in K[x]$ and $s\geq 1$} & $\Gamma(Q_\infty,Q_1,\dots,Q_s)$ \\ \hline
	\end{tabular}
	\label{table2}
    \end{table}
    
    In this work, our main interest is the study of a more general Kummer extension defined by \eqref{kummer_extension}, where $f(x)\in K(x)$ is not required to completely split over $K$.
	Moreover, we remove the assumption that $K$ contains a primitive $m$-th root of unity.
	To derive results adapted to this generalized Kummer extension, we establish Lemma \ref{D_dimension} based on the result for classical Kummer extensions given in \cite[Theorem 2.2]{maharajCodeConstruction2004}.
    Then inspired by the study of \cite{castellanosWeierstrassSemigroupsPure2024}, we describe  the set of gaps at a totally ramified place of degree one, and our description provides a unified formulation.
    As a consequence, for a totally ramified place $Q$ of degree one, we explicitly provide a system of generators, as well as the multiplicity and the Frobenius number, of the Weierstrass semigroup $H(Q)$.
	Moreover, by employing the relationship between Weil differentials and gaps, we obtain the following necessary and sufficient condition for $H(Q)$ to be symmetric.
    \begin{theorem*}[see Theorem \ref{symmetric_thm}]
       Suppose that $\gcd(m,\lambda_s) = d_s = 1$ for some $0\leq s\leq r$. Then the Weierstrass semigroup $H(Q_s)$ at the totally ramified place $Q_s$ is symmetric if and only if there exists $1\leq u\leq m-1$ such that 
       $$u\lambda_i\equiv\gcd(m,\lambda_i)\pmod m$$
       for all $i\in\{0,1,\dots,r\}\setminus\{s\}$.
    \end{theorem*}
    The above theorem greatly extends the results of \cite[Theorems A and B]{mendozaKummerExtensions2023} and \cite[Corollary 3.8]{cotterillGapSets2025}, which completely resolve the symmetry criterion problem for totally ramified rational places on Kummer extensions.
	
	As another main contribution, we explicitly provide the set $\Gamma(Q_{\ell_0},\dots,Q_{\ell_s})$ at $s$ totally ramified places of degree one on the function field defined by \eqref{kummer_extension} (see Theorem \ref{thm_minimal_set}).
	Compared with the results of \cite{yangWeierstrassSemigroups2017}, our work constitutes a substantial generalization. We not only propose a unified computational formulation that avoids the separate discussion of the infinite place, but also give a rigorous proof on a more general Kummer extension.
	In contrast to the related results given in \cite[Corollary 3.5]{castellanosGeneralizedWeierstrass2026}, we adopt a different approach introduced by Matthews in \cite{matthewsWeierstrassSemigroup2004,matthewsWeierstrassSemigroups2005}.
	Furthermore, we give explicit functions whose pole divisors have  coefficients lying in the set $\Gamma(Q_{\ell_0},\dots,Q_{\ell_s})$, together with a necessary and sufficient condition for $\Gamma(Q_{\ell_0},\dots,Q_{\ell_s})$ to be nonempty. Neither of these ingredients appears in \cite{castellanosGeneralizedWeierstrass2026}.
	Lastly, we illustrate applications of our results to generalized Kummer extensions with equal multiplicities, to function fields of GGS curves, and to function fields of subcovers of the BM curve.

    This paper is organized as follows. In Section \ref{section2}, we briefly recall some notations and preliminary results related to function fields and Weierstrass semigroups.
    In Section \ref{section3}, for a totally ramified place $Q$ of degree one on the function field defined by \eqref{kummer_extension}, we provide an explicit description of the gap set $G(Q)$ and the Weierstrass semigroup $H(Q)$, and then establish a necessary and sufficient condition for $H(Q)$ to be symmetric.
    In Section \ref{section4}, we focus on determining the minimal generating sets of the Weierstrass semigroups at several totally ramified places of degree one on the function field defined by \eqref{kummer_extension}.
	In Section \ref{section5}, we exhibit some examples, including generalized Kummer extensions with equal multiplicities, function fields of GGS curves, and function fields of subcovers of the BM curve.

\section{Preliminaries}\label{section2}
    Throughout this article, let $q$ be a prime power and let $\mathbb{F}_q$ be the finite field with $q$ elements. Let $K$ be an algebraic extension of $\mathbb{F}_q$.
    For $a,b\in\Z$, we denote by $\gcd(a,b)$ the greatest common divisor of $a$ and $b$. Let $\N := \{1,2,3,\dots\}$ and $\N_0 := \{0,1,2,\dots\}$.
    For $c\in \mathbb{R}$, we denote by $\floor{c}$ the largest integer not greater than $c$ and by $\ceil{c}$ the smallest integer not less than $c$. 

	\begin{lemma}\label{ceil_and_floor}\cite[Lemma 4.1]{castellanosWeierstrassSemigroupsPure2024} Let $a$ and $b$ be elements in $\mathbb{R}$. The following statements hold:

	(i) $\floor*{-a} = -\ceil*{a}$.
	
	(ii) $\ceil*{a}-\floor*{a} = \begin{cases} 0, & \text{if } a \in \mathbb{Z}, \\ 1, & \text{if } a \notin \mathbb{Z}. \end{cases}$
	
	(iii) If $a$ and $b$ are positive integers, then
	$$\sum_{k=1}^{b-1} \left\lfloor \frac{ka}{b} \right\rfloor = \frac{(a-1)(b-1) + \gcd(a,b) - 1}{2}.$$
	\end{lemma}

\subsection{Function fields}
Let $F/K$ be a function field with constant field $K$. Let $g$ be the genus of $F$. We denote by $\P_F$ the set of places of $F$, by $\Omega_F$ the module of Weil differentials of $F$, by $v_P$ the discrete valuation of $F$ with respect to the place $P\in \P_F$, and by $\Div (F)$ the free abelian group generated by the places in $F$. A place of degree one is called a rational place of $F$. An element $D\in \Div(F)$ is called a divisor of $F$ and its degree is given by $\deg D := \sum_{P\in{\operatorname{supp} D}} v_P(D)\cdot \deg P$, where $\operatorname{supp} D$ is the support of $D$.
For a non-zero element $z\in F$, we denote by $(z)_F$, $(z)_\infty$, and $(z)_0$ the principal divisor, the pole divisor and the zero divisor of $z$, respectively.
For a non-zero element $\omega\in\Omega_F$, we denote by $(\omega)_F$ the canonical divisor corresponding to $\omega$.

Given a divisor $D\in \Div(F)$, we have the following two vector spaces over $K$: the Riemann-Roch space 
$$\mathcal{L}(D) := \{z\in F\mid (z)_F \geq -D\}\cup\{0\}$$
and the space of Weil differentials 
$$\Omega_F(D) := \{\omega\in\Omega_F\mid (\omega)_F\geq D\}\cup\{0\}.$$
We denote by $\ell(D)$ the dimension of $\mathcal{L}(D)$ over $K$. A non-zero Weil differential $\omega\in \Omega_F$ is said to be regular if $\omega\in \Omega_F(0)$.
Since $\Omega_F$ is a one-dimensional vector space over $F$, it follows that $(z\omega)_F = (z)_F + (\omega)_F$ is also a canonical divisor for any non-zero elements $z\in F$ and $\omega\in \Omega_F$.
For each divisor $D\in \Div(F)$, we have $\ell(D) = 0$ if $\deg D<0$, and the Riemann-Roch Theorem implies that 
$\ell(D) = \deg D + 1 - g$ if $\deg D \geq 2g-1$.

Next we introduce the notion of Weierstrass semigroups in $F$. 
Let $P_1,\dots,P_s$ be $s$ rational places of $F$. The Weierstrass semigroup at $P_1,\dots,P_s$ is defined by 
$$H(P_1,\dots,P_s) := \left\{(n_1,\dots,n_s)\in \N_0^s~|~\exists z\in F \text{~with~}(z)_\infty = \sum_{i=1}^s n_iP_i\right\}.$$
The complementary set $G(P_1,P_2,\dots,P_s):= \mathbb{N}_0^s \setminus H(P_1,\dots, P_s)$ is called the set of gaps at $P_1,P_2,\dots,P_s$. 
An element in $G(P_1,P_2,\dots,P_s)$ is called a gap at $P_1,P_2,\dots,P_s$.

Let $P$ be a rational place of $F$. The Weierstrass semigroup $H(P)$ at a single rational place has additional properties. If $g \geq 1$, then $G(P) = \N_0\setminus H(P)$ contains exactly $g$ gaps $1 = a_1<a_2<\dots <a_g\leq 2g-1$ at $P$.
The smallest non-zero element of $H(P)$ is called the multiplicity of $H(P)$ and is denoted by $m_{H(P)}$,
while the largest element of $G(P)$ is called the Frobenius number of $H(P)$ and is denoted by $F_{H(P)}$. We say that $H(P)$ is symmetric if $F_{H(P)} = 2g-1$.
There are two useful lemmas concerning $H(P)$ and $G(P)$.
\begin{lemma}\cite[Remark 4.4]{kirfelMinimumDistance1995a}\label{symmetric_to_canonical}
    The Weierstrass semigroup $H(P)$ is symmetric if and only if $W = (2g-2)P$ is a canonical divisor.
\end{lemma}
\begin{lemma}\cite[Corollary 14.2.5]{villasalvadorTopicsTheory2006}\label{differential_gap}
    Let $a\in\mathbb{N}$. Then $a\in G(P)$ if and only if there exists a regular Weil differential $\omega\in\Omega_F$ such that $v_P(\omega) = a-1$.
\end{lemma}

\subsection{The minimal generating sets of Weierstrass semigroups}

To describe the minimal generating sets of Weierstrass semigroups at several rational places, we first introduce additional notation. For two elements $\mathbf{a} = (a_1,\dots,a_s), \mathbf{b} = (b_1,\dots,b_s)\in \N_0^s$, we define a partial order $\preceq$ on $\N_0^s$ by $\mathbf{a}\preceq \mathbf{b}$ if and only if $a_i\leq b_i$ for all $1\leq i\leq s$.
Furthermore, if $a_i<b_i$ for some $1\leq i\leq s$, we write $\mathbf{a}\prec \mathbf{b}$. Let $S\subseteq \N_0^s$ and $\mathbf{a} \in S$. We say that $\mathbf{a}$ is minimal in $S$ if $\mathbf{b}\not\preceq \mathbf{a}$ for all $\mathbf{b}\in S\setminus\{\mathbf{a}\}$. We denote by $|K|$ the cardinality of $K$. If $K$ is an infinite field, we adopt the convention that $a<|K|$ for every integer $a$.

For a function field $F/K$ with genus $g>0$, let $P_1,\dots,P_s$ be $s$ distinct rational places of $F$, where $1\leq s\leq |K|$. From the work of Carvalho and Torres \cite{carvalhoGoppaCodes2005}, the dimensions of Riemann-Roch spaces can be used to characterize $H(P_1,P_2,\dots,P_s)$ and $G(P_1,P_2,\dots,P_s)$. Given an $s$-tuple $\n = (n_1,\dots,n_s)\in \N_0^s$, we have that $\n\in H(P_1,\dots,P_s)$ if and only if 
$$\ell\left(\sum_{i=1}^s n_iP_i\right) = \ell\left(\sum_{i=1}^s n_iP_i-P_j\right)+1 \text{ for all } 1\leq j\leq s.$$
Moreover, we have that $\n\in G(P_1,\dots,P_s)$ if and only if 
$$\ell\left(\sum_{i=1}^s n_iP_i\right) = \ell\left(\sum_{i=1}^s n_iP_i-P_j\right) \text{ for some } 1\leq j\leq s.$$

Next, we introduce the definition of $\Gamma(P_1,\dots,P_t)$ for $1\leq t\leq s$, which is proposed by Matthews \cite{matthewsWeierstrassSemigroup2004}. Let $\Gamma(P_1) := H(P_1)$. For $s\geq 2$, define $\Gamma(P_1,\dots,P_s)$ by
$$\{\a\in \N_0^s\mid \a \text{ is minimal in }\{\n\in H(P_1,\dots,P_s)~|~n_i = a_i\} \text{ for some } 1\leq i\leq s\}.$$
\begin{proposition}\cite[Proposition 3]{matthewsWeierstrassSemigroup2004}\label{minimal_for_all}
	Let $\a = (a_1,\dots,a_s)\in\N_0^s$. Then $\a$ is minimal in $\{\n\in H(P_1,\dots,P_s)\mid n_i = a_i\}$ for some $1\leq i\leq s$,
	if and only if $\a$ is minimal in $\{\n\in H(P_1,\dots,P_s)\mid n_i = a_i\}$ for all $1\leq i\leq s$.
\end{proposition}

\begin{lemma}\cite[Lemma 2.6]{castellanosWeierstrassSemigroup2018}\label{discrepancy}
	Let $\n = (n_1, \dots, n_s) \in H(P_1, \dots, P_s)$ and $A = n_1P_1 + \dots + n_sP_s$. Then $\n \in \Gamma(P_1, \dots, P_s)$ if and only if 
    $$\ell(A) = \ell(A - P) + 1 = \ell(A - P - Q)  + 1\text{ and } \ell(A) = \ell(A - Q) + 1 = \ell(A - P - Q) + 1$$
	for any two places $ P, Q \in \{P_1, \dots, P_s\} $.
\end{lemma}

\begin{lemma}\cite[Lemma 4]{matthewsWeierstrassSemigroup2004}\label{lemma5}
	Suppose that $s\geq 2$. Then 
	$$\Gamma(P_1,\dots,P_s) \subseteq G(P_1)\times\dots\times G(P_s).$$
\end{lemma}

For $s=2$, suppose that $G(P_1) = \{a_1<a_2<\dots<a_g\}$ and $G(P_2) = \{b_1<b_2<\dots<b_g\}$.
For each gap $a_i$ at $P_1$, let $n_{a_i} = \min\{b\in\N_0\mid(a_i,b)\in H(P_1,P_2)\}$.
From \cite[Lemma 2.6]{kimIndexWeierstrass1994}, we have the equality $\{n_a \mid a \in G(P_1)\} = G(P_2)$, and therefore there exists a permutation $\tau$ of $\{1, 2, \dots, g\}$ such that $n_{a_i} = b_{\tau(i)}$.
The graph of the bijection between $G(P_1)$ and $G(P_2)$ defining the permutation $\tau$ is the set $\Gamma(P_1, P_2) = \{(a_i, b_{\tau(i)}) \mid i = 1, \dots, g\}$.
The following lemma characterizes it.
\begin{lemma}\cite[Lemma 2]{hommaWeierstrassSemigroup1996}\label{minimal_generating_two}
	Let $\Gamma$ be a subset of $(G(P_1) \times G(P_2)) \cap H(P_1, P_2)$. 
	If there exists a permutation $\tau$ of $\{1, 2, \dots, g\}$ such that $\Gamma = \{(a_i, b_{\tau(i)}) \mid i = 1, \dots, g\}$, then $\Gamma = \Gamma(P_1, P_2)$.
\end{lemma}

Let $1\leq t\leq s$ and $I = \{i_1,\dots,i_t\}\subseteq \{1,\dots,s\}$. Define the natural inclusion
$$\begin{array}{ccccc}
	\iota_I:& \mathbb{N}_0^{t} & \longrightarrow & \mathbb{N}_0^s,\\
	&(n_{i_1},\dots,n_{i_t}) & \longmapsto &(n_1,\dots,n_s),
\end{array}$$
where $n_{j} = 0$ for $j\not\in I$, and define the natural projection
$$\begin{array}{ccccc}
	\pi_I:& \mathbb{N}_0^{s} & \longrightarrow & \mathbb{N}_0^t,\\
	&(n_1,\dots,n_s) & \longmapsto &(n_{i_1},\dots,n_{i_t}).
\end{array}$$
The minimal generating set of $H(P_1,\dots,P_s)$ is defined as
$$\tilde{\Gamma}(P_1,\dots,P_s) := \bigcup_{t=1}^s\bigcup_{\substack{I = \{i_1,\dots,i_t\}\\ 1\leq i_1<\dots<i_t\leq s}} \iota_I(\Gamma(P_{i_1},\dots,P_{i_t})).$$
Given $\mathbf{u_1},\dots,\mathbf{u_t}\in \N_0^s$, where $t\geq 2$, define the least upper bound of $\mathbf{u_1},\dots,\mathbf{u_t}$ by
$$\lub\{\mathbf{u_1},\dots,\mathbf{u_t}\} := (\max\{u_{1_1},\dots,u_{t_1}\},\dots,\max\{u_{1_s},\dots,u_{t_s}\}).$$
The following theorem shows that $H(P_1,\dots,P_s)$ is determined by $\tilde{\Gamma}(P_1,\dots,P_s)$.
\begin{theorem}\cite[Theorem 7]{matthewsWeierstrassSemigroup2004}\label{lub_semi_group}
	Suppose that $s\geq 2$. Then 
	$$H(P_1,\dots,P_s) = \{\lub\{\mathbf{u_1},\dots,\mathbf{u_s}\}~|~\mathbf{u_1},\dots,\mathbf{u_s}\in \tilde{\Gamma}(P_1,\dots,P_s)\}.$$
\end{theorem}

\subsection{Generalized Kummer extensions}
Let $m\geq 2$ be an integer with $\gcd(m,q)=1$, and let $r\geq 1$ be an integer. Let $f(x)\in K(x)$ be such that for every $d\mid m$ with $d\geq 2$, $f(x)$ is not a $d$-th power of any element in $\bar{K}(x)$, where $\bar{K}$ is the algebraic closure of $K$.
Consider the function field extension $F = K(x,y)/K(x)$ defined by the equation \eqref{kummer_extension}:
    $$y^m = f(x) = \alpha\cdot\prod_{i=1}^r p_i(x)^{\lambda_i},$$
where $\alpha\in K\setminus\{0\}$, each $\lambda_i\in \Z\setminus\{0\}$, and $p_1(x),\dots,p_r(x)\in K[x]$ are pairwise distinct monic irreducible polynomials. Note that $\varphi(T) = T^m - f(x)\in K(x)[T]$, which is the minimal polynomial of $y$ over $K(x)$, is also irreducible in $\bar{K}(x)[T]$. 
It follows from \cite[Corollary 3.6.8]{stichtenothAlgebraicFunctionFields2009} that $K$ is the full constant field of $F$.

Let $d_i:= \deg\, p_i(x)$ for $1\leq i\leq r$.
Let $\lambda_0:=-\sum_{i=1}^{r}\lambda_i d_i$ and $d_0:=1$.
By \cite[Proposition 3.7.3]{stichtenothAlgebraicFunctionFields2009}, the genus of $F$ is given by
    $$g = \frac{2-2m+\sum_{i=0}^{r} \left(m-\gcd(m,\lambda_i)\right) \cdot d_i}{2}.$$
For each $1\leq i\leq r$, let $P_i$ and $P_\infty$ be the places in $\P_{K(x)}$ corresponding to the zero of $p_i(x)$ and the pole of $x$, respectively.
If $\gcd(m,\lambda_i) = 1$, we denote by $Q_i$ the only place in $F$ lying over $P_i$ and say that $Q_i$ is totally ramified.
If $\gcd(m,\lambda_{0}) = 1$, we denote by $Q_\infty$ the only place lying over $P_\infty$ and say that  $Q_\infty$ is totally ramified.

For convenience, we also denote by $P_0$ the pole of $x$ in $\mathbf{P}_{K(x)}$, and by $Q_0$ the only place lying over $P_0$ if $\gcd(m,\lambda_0)=1$.
If $d_i = \gcd(m,\lambda_i) = 1$ for some $0\leq i\leq r$, then $Q_i$ is a totally ramified place of degree one, which is our main object of study.

We set $z_0 := 1$ and  $z_i := p_i(x)$ for $1\leq i\leq r$. Then we have the following principal divisors:
\begin{equation}\label{divisor1}
	(z_i)_{F}=\frac{m}{\gcd(m,\lambda_i)}\displaystyle\sum_{Q\in \P_{F}, \, Q|P_{i}}Q-\frac{md_i}{\gcd(m,\lambda_0)}\sum_ {Q\in \P_{F}, \, Q|P_{\infty}}Q \text{ for } 0\leq i\leq r,
\end{equation}
\begin{equation}\label{divisor2}
	(y)_{F}=\displaystyle\sum_{i=0}^{r}\frac{\lambda_i}{\gcd(m,\lambda_i)}\sum_{{ Q\in \P_{F}, \, Q|P_{i}}}Q.
\end{equation}
The different of $F/K(x)$ is 
$${\rm Diff}(F/K(x)) =  \sum_{i=0}^{r}\left(\frac{m}{\gcd(m,\lambda_i)}-1\right)\sum_{Q\in \P_{F},\,Q\mid P_i}Q.$$
Thus, there exists a canonical divisor $W$ given by 
\begin{equation}\label{divisor3}
\begin{split}
W &= -2(x)_\infty  + {\rm Diff}(F/K(x))\\
& = -\left(\frac{m}{\gcd(m,\lambda_0)}+1\right)\sum_{Q\in \P_{F},\,Q\mid P_0}Q  +\sum_{i=1}^{r}\left(\frac{m}{\gcd(m,\lambda_i)}-1\right)\sum_{Q\in \P_{F},\,Q\mid P_i}Q.
\end{split}
\end{equation}

For any divisor $D$ of $F$, write
$D = \sum_{P\in \P_{K(x)}}\sum_{Q\in \P_F,\,Q|P}\, n_Q\, Q$.
We define the restriction of $D$ to $K(x)$ as
$$ D\Big|_{K(x)} := \sum_{P\in \P_{K(x)}} \min\left\{\left\lfloor\frac{n_Q}{e(Q|P)}\right\rfloor\colon{Q|P}\right\} P,$$
where $e(Q|P)$ is the ramification index of $Q$ over $P$. 
If $K$ contains a primitive $m$-th root of unity, then the extension $F/K(x)$ is a Kummer extension and its Galois group is cyclic. The following result was given by Maharaj in \cite{maharajCodeConstruction2004}.
\begin{theorem}{\cite[Theorem 2.2]{maharajCodeConstruction2004}} \label{ThMaharaj}
	Suppose that $K$ contains a primitive $m$-th root of unity. Then for any divisor $D$ of $F$ that is invariant under the action of the Galois group $\operatorname{Gal}(F/K(x))$,
	$$ \mathcal{L}(D)= \bigoplus\limits_{t=0}^{m-1} \mathcal L\left(\left[D+(y^t)_F\right]\Big|_{K(x)}\right)\,y^t.$$
	\end{theorem}

\begin{remark}
    In the classical framework, the extension $F/K(x)$ is said to be a Kummer extension under the assumption that $K$ contains a primitive $m$-th root of unity. In this paper, we remove this assumption and call $F/K(x)$ a generalized Kummer extension.
\end{remark}

\section{The Weierstrass semigroup at a totally ramified place of degree one on generalized Kummer extensions}\label{section3}
In this section, we consider the generalized Kummer extension $F = K(x,y)/K(x)$ defined by the equation \eqref{kummer_extension}.
First, we present some required lemmas.
Then we give a unified description of the set of gaps $G(Q)$ at any totally ramified place $Q$ of degree one. 
Furthermore, we explicitly provide a system of generators, the multiplicity, and the Frobenius number of $H(Q)$.
Finally, we determine a necessary and sufficient condition for $H(Q)$ to be symmetric.

\subsection{Some required lemmas}
In this subsection, we first state several lemmas that will be used in the later proofs.
\begin{lemma}\label{sum_mod}
	Suppose that $(a_0,\dots,a_r)\in \Z^{r+1}$, $b, c\in\Z$ and $b \equiv c\pmod m$. Then
	\begin{equation*}
	\sum_{i=0}^r\left\lfloor\frac{a_i+b\lambda_i}{m}\right\rfloor d_i = \sum_{i=0}^r\floor*{\frac{a_i+c\lambda_i}{m}} d_i  \text{~~and~~} 
	\sum_{i=0}^r\ceil*{\frac{a_i+b\lambda_i}{m}} d_i = \sum_{i=0}^r\ceil*{\frac{a_i+c\lambda_i}{m}} d_i.
	\end{equation*}
\end{lemma}
\begin{proof}
	Suppose that $b = c + km$ for some $k\in \Z$. It follows from $\sum_{i=0}^r \lambda_i d_i = 0$ that
\begin{align*}
\sum_{i=0}^r \left\lfloor\frac{a_i+b\lambda_i}{m}\right\rfloor d_i
	&= \sum_{i=0}^r \left\lfloor\frac{a_i+(c+km)\lambda_i}{m}\right\rfloor d_i\\
	&= \sum_{i=0}^r \left\lfloor\frac{a_i+c\lambda_i}{m}\right\rfloor d_i +k\sum_{i=0}^{r}\lambda_id_i 
	 = \sum_{i=0}^r \left\lfloor\frac{a_i+c\lambda_i}{m}\right\rfloor d_i.
\end{align*}
Similarly, we have $\sum_{i=0}^r\ceil*{\frac{a_i+b\lambda_i}{m}} d_i = \sum_{i=0}^r\ceil*{\frac{a_i+c\lambda_i}{m}} d_i$.
\end{proof}

Next we present a lemma similar to part (iii) of Lemma \ref{ceil_and_floor}.
\begin{lemma}\label{ceil_sum}
	For $\lambda\in\Z$,
	$$\sum_{i=1}^{m-1}\ceil*{\frac{i\lambda}{m}} = \frac{m(\lambda+1)-\lambda-\gcd(m,\lambda)}{2}.$$
\end{lemma}
\begin{proof}
	If $\lambda = 0$, one verifies that
	$$\sum_{i=1}^{m-1}\ceil*{\frac{i\lambda}{m}} = 0 =  \frac{m(\lambda+1)-\lambda-\gcd(m,\lambda)}{2}.$$
	If $\lambda>0$, since $\#\{1 \leq i \leq m-1 \mid m \text{ divides }i\lambda\}=\gcd(m, \lambda)-1$, by Lemma \ref{ceil_and_floor}, we have
	\begin{align*}
		\sum_{i=1}^{m-1}\ceil*{\frac{i\lambda}{m}} &= m-\gcd(m, \lambda)+\sum_{i=1}^{m-1}\floor*{\frac{i\lambda}{m}} \\
		&=  m-\gcd(m, \lambda)+\frac{(m-1)(\lambda-1)+\gcd(m,\lambda)-1}{2}\\
		& =  \frac{m(\lambda+1)-\lambda-\gcd(m,\lambda)}{2}.
	\end{align*}
	If $\lambda<0$, by Lemma \ref{ceil_and_floor}, we have
	\begin{align*}
		\sum_{i=1}^{m-1}\ceil*{\frac{i\lambda}{m}} = - \sum_{i=1}^{m-1}\floor*{\frac{-i\lambda}{m}} &= -\frac{(m-1)(-\lambda-1)+\gcd(m,\lambda)-1}{2}\\
		& =  \frac{m(\lambda+1)-\lambda-\gcd(m,\lambda)}{2}.
	\end{align*}
	The discussion of the above three cases completes the proof.
\end{proof}

    The following lemma is derived from the theory of constant field extensions of function fields and Theorem \ref{ThMaharaj}. Compared with Theorem \ref{ThMaharaj}, our formula for generalized Kummer extensions removes the requirement that $K$ contains a primitive $m$-th root of unity.
\begin{lemma}\label{lemma:D_sumdim}
    Suppose that $K'$ is an algebraic extension of $K$ such that $K'$ contains a primitive $m$-th root of unity. Let $F' = FK'$ be the constant field extension of $F$, and let $\operatorname{Con}_{F'/F}(\cdot)$ be the conorm with respect to $F'/F$.
    Let $D$ be a divisor of $F$ such that $\operatorname{Con}_{F'/F}(D)$ is invariant under the action of the Galois group $\operatorname{Gal}(F'/K'(x))$. Then
    \begin{equation*}\label{equation:D_dim}
        \ell(D) = \sum_{t=0}^{m-1}\ell\left(\left[D+(y^t)_F\right]\Big|_{K(x)}\right).
    \end{equation*}
\end{lemma}
\begin{proof}
    Let $A$ be a divisor of $F$. Let $\operatorname{Con}_{K'(x)/K(x)}(\cdot)$ be the conorm with respect to $K'(x)/K(x)$. According to \cite[Theorem 3.6.3]{stichtenothAlgebraicFunctionFields2009}, we have $\ell(\operatorname{Con}_{K'(x)/K(x)}(A|_{K(x)})) = \ell(A|_{K(x)})$. Moreover, we can write 
    $$A = \sum_{P\in \P_{K(x)}}\sum_{Q\in \P_F,\,Q|P} n_Q\, Q \text {~~and~~} \operatorname{Con}_{F'/F}(A) = \sum_{P\in \P_{K(x)}}\sum_{Q\in \P_F,\,Q|P}\, n_Q \sum_{Q'\in\P_{F'}\, Q'|Q} Q'.$$
    We obtain that 
    \begin{align*}
        \left[\operatorname{Con}_{F'/F}(A)\right]\Big|_{K'(x)} = \sum_{P\in \P_{K(x)}}\sum_{P'\in \P_{K'(x)},\,P'|P}\min\left\{\left\lfloor\frac{n_Q}{e(Q'|P')}\right\rfloor\colon{Q'|P'}\right\} P',
    \end{align*}
    where $Q\in \P_F$ with $Q'|Q$. For each $Q'|P'|P$, we have $e(Q'|P') = e(Q|P)$ since $e(Q'|Q) = 1 = e(P'|P)$.
    Thus 
    \begin{align*}
        \left[\operatorname{Con}_{F'/F}(A)\right]\Big|_{K'(x)} &= \sum_{P\in \P_{K(x)}}\sum_{P'\in \P_{K'(x)},\,P'|P}\min\left\{\left\lfloor\frac{n_Q}{e(Q|P)}\right\rfloor\colon{Q'|P'}\right\} P'\\
        &= \sum_{P\in \P_{K(x)}}\min\left\{\left\lfloor\frac{n_Q}{e(Q|P)}\right\rfloor\colon{Q|P}\right\} \sum_{P'\in \P_{K'(x)},\, P'|P}P' \\
        & = \operatorname{Con}_{K'(x)/K(x)}\left(A|_{K(x)}\right).
    \end{align*}
    Now setting $A = D+(y^t)_{F}$, we get that
    \begin{align*}
        \ell\left(\left[\operatorname{Con}_{F'/F}(D)+(y^t)_{F'}\right]\Big|_{K'(x)}\right) &= \ell\left(\operatorname{Con}_{K'(x)/K(x)}\left(\left[D+(y^t)_F\right]\Big|_{K(x)}\right)\right)\\
        & = \ell\left(\left[D+(y^t)_F\right]\Big|_{K(x)}\right)
    \end{align*}
    for each $0\leq t\leq m-1$.
    By Theorem \ref{ThMaharaj}, we have
    $$\ell(\operatorname{Con}_{F'/F}(D)) = \sum_{t=0}^{m-1}\ell\left(\left[\operatorname{Con}_{F'/F}(D)+(y^t)_{F'}\right]\Big|_{K'(x)}\right) =  \sum_{t=0}^{m-1}\ell\left(\left[D+(y^t)_{F}\right]\Big|_{K(x)}\right).$$
    Then the result of the lemma follows from $\ell(\operatorname{Con}_{F'/F}(D)) = \ell(D)$.
 \end{proof}

From the above lemma, we immediately obtain the following corollary.
\begin{corollary}\label{D_dimension}
	Let $\{\ell_0,\ell_1,\dots,\ell_r\}$ be a permutation of $\{0,1,\dots,r\}$ and  $0\leq s \leq r$. Suppose that $\gcd(m,\lambda_{\ell_i})=1$ for all $0\leq i \leq s$, and let $(a_0,\dots,a_s)\in\mathbb{Z}^{s+1}$. Then
	$$\ell\left(\sum_{i=0}^{s} a_i Q_{\ell_i}\right) = \sum_{t=0}^{m-1}\max\left\{0,\sum_{i=0}^{s}\floor*{\frac{a_i+t\lambda_{\ell_i}}{m}} d_{\ell_i} + \sum_{i=s+1}^{r}\floor*{\frac{t\lambda_{\ell_i}}{m}} d_{\ell_i}+1\right\}.$$
\end{corollary}
\begin{proof}
	Note that $Q_{\ell_0},Q_{\ell_1},\dots,Q_{\ell_s}$ are totally ramified places in the Kummer extension $F/K(x)$ since $\gcd(m,\lambda_{\ell_i})=1$ for all $0\leq i \leq s$. The principal divisor of $y$ in $F$ is 
	$$(y)_F = \sum_{i=0}^{s} \lambda_{\ell_i} Q_{\ell_i} + \sum_{i=s+1}^{r} \sum_{Q\in \mathbf{P}_F,\,Q\mid P_{\ell_i}} \frac{\lambda_{\ell_i}}{\gcd(m,\lambda_{\ell_i})} Q.$$
	Then for each $t\in\{0,1,\dots,m-1\}$, 
	$$\sum_{i=0}^{s} a_i Q_{\ell_i} + (y^t)_F = \sum_{i=0}^{s}(a_i + t\lambda_{\ell_i})Q_{\ell_i} + \sum_{i=s+1}^{r} \sum_{Q\in \mathbf{P}_F,\,Q\mid P_{\ell_i}} \frac{t\lambda_{\ell_i}}{\gcd(m,\lambda_{\ell_i})} Q.$$
	It follows that 
	$$\left(\sum_{i=0}^{s} a_i Q_{\ell_i} + (y^t)_F\right)\Big|_{K(x)} = \sum_{i=0}^{s}\floor*{\frac{a_i + t\lambda_{\ell_i}}{m}}P_{\ell_i} + \sum_{i=s+1}^{r}\floor*{\frac{t\lambda_{\ell_i}}{m}} P_{\ell_i}.$$
	By Lemma \ref{lemma:D_sumdim}, we obtain
	$$\ell\left(\sum_{i=0}^{s} a_i Q_{\ell_i}\right) = \sum_{t=0}^{m-1}\ell\left(\sum_{i=0}^{s}\floor*{\frac{a_i + t\lambda_{\ell_i}}{m}}P_{\ell_i} + \sum_{i=s+1}^{r}\floor*{\frac{t\lambda_{\ell_i}}{m}} P_{\ell_i}\right).$$
	Since the genus of $K(x)$ is 0, it follows from Riemann-Roch Theorem that 
	\begin{align*}
		&\ell\left(\sum_{i=0}^{s}\floor*{\frac{a_i + t\lambda_{\ell_i}}{m}}P_{\ell_i} + \sum_{i=s+1}^{r}\floor*{\frac{t\lambda_{\ell_i}}{m}} P_{\ell_i}\right)\\
		=&\max\left\{0, \deg \left(\sum_{i=0}^{s}\floor*{\frac{a_i + t\lambda_{\ell_i}}{m}}P_{\ell_i} + \sum_{i=s+1}^{r}\floor*{\frac{t\lambda_{\ell_i}}{m}} P_{\ell_i}\right) + 1 -0\right\}\\
		=&\max\left\{0,\sum_{i=0}^{s}\floor*{\frac{a_i+t\lambda_{\ell_i}}{m}} d_{\ell_i} + \sum_{i=s+1}^{r}\floor*{\frac{t\lambda_{\ell_i}}{m}} d_{\ell_i}+1\right\}.
	\end{align*}
	Thus we obtain that 
	$$\ell\left(\sum_{i=0}^{s} a_i Q_{\ell_i}\right) = \sum_{t=0}^{m-1}\max\left\{0,\sum_{i=0}^{s}\floor*{\frac{a_i+t\lambda_{\ell_i}}{m}} d_{\ell_i} + \sum_{i=s+1}^{r}\floor*{\frac{t\lambda_{\ell_i}}{m}} d_{\ell_i}+1\right\}.$$
\end{proof}

\subsection{A unified description of $G(Q_s)$}
In this subsection, we give a unified description of $G(Q_s)$ for a totally ramified place of degree one in the generalized Kummer extension $F/K(x)$. First, we characterize $G(Q_s)$ as follows.
\begin{proposition}\label{describe_gap_set}
	Suppose that $d_s = \gcd(m,\lambda_s)=1$ for some $0\leq s\leq r$. Let $\lambda \in \mathbb{Z} $ be the inverse of $\lambda_s$ modulo m.
	Then
	$$G(Q_s) = \left\{a\in \mathbb{N} ~\Big |~ \sum_{i=0}^r\left\lfloor\frac{-a\lambda\lambda_i}{m}\right\rfloor d_i + \left\lceil\frac{a}{m}\right\rceil \leq -1\right\}.$$
\end{proposition}
\begin{proof}
	Let $a\in \mathbb{N}$. 
	We have $a\in G(Q_s)$ if and only if $\ell(aQ_s) = \ell((a-1)Q_s)$. By Corollary \ref{D_dimension}, the equality $\ell(aQ_s) = \ell((a-1)Q_s)$ is equivalent to the condition that for each $t\in \{0,1,\dots,m-1\}$, either 
	$$\left\lfloor\frac{a+t\lambda_{s}}{m}\right\rfloor + \sum_{i=0,i\neq s}^{r}\left\lfloor\frac{t\lambda_{i}}{m}\right\rfloor d_{i} \leq -1$$
	or 
	$$\left\lfloor\frac{a+t\lambda_{s}}{m}\right\rfloor + \sum_{i=0,i\neq s}^{r}\left\lfloor\frac{t\lambda_{i}}{m}\right\rfloor d_{i}\geq 0 \text{ and } \left\lfloor\frac{a+t\lambda_{s}}{m}\right\rfloor = \left\lfloor\frac{a+t\lambda_{s}-1}{m}\right\rfloor.$$
	Note that $\floor*{\frac{a+t\lambda_{s}}{m}}\neq \floor*{\frac{a+t\lambda_{s}-1}{m}}$ if and only if $t\equiv -a\lambda \pmod m$. Thus $a\in G(Q_s)$ if and only if
	$$\left\lfloor\frac{a+t\lambda_{s}}{m}\right\rfloor + \sum_{i=0,i\neq s}^{r}\left\lfloor\frac{t\lambda_{i}}{m}\right\rfloor d_{i} \leq -1$$
	for the unique $t\in\{0,1,\dots,m-1\}$ satisfying $t\equiv -a\lambda \pmod m$. By Lemma \ref{sum_mod}, we obtain that 
	\begin{align*}
		\left\lfloor\frac{a+t\lambda_{s}}{m}\right\rfloor + \sum_{i=0,i\neq s}^{r}\left\lfloor\frac{t\lambda_{i}}{m}\right\rfloor d_{i} = \left\lfloor\frac{a-a\lambda\lambda_{s}}{m}\right\rfloor + \sum_{i=0,i\neq s}^{r}\left\lfloor\frac{-a\lambda\lambda_{i}}{m}\right\rfloor d_{i}
	\end{align*}
	Let $\lambda\lambda_s = 1+ km$ for some $k\in \mathbb{Z}$. Then $\floor*{\frac{a-a\lambda\lambda_s}{m}} = -ka$ and $\floor*{\frac{-a\lambda\lambda_s}{m}} = -ka -\ceil*{\frac{a}{m}}$. 
	Thus, we get that $\floor*{\frac{a-a\lambda\lambda_s}{m}} = \floor*{\frac{-a\lambda\lambda_s}{m}} + \ceil*{\frac{a}{m}}$.
	This yields $a\in G(Q_s)$ if and only if 
	$$\sum_{i=0}^r\left\lfloor\frac{-a\lambda\lambda_i}{m}\right\rfloor d_i + \left\lceil\frac{a}{m}\right\rceil \leq -1.$$
\end{proof}

Now we explicitly describe the set of gaps at a totally ramified place of degree one in a unified way.
\begin{proposition}\label{gap_set}
Suppose that $d_s = \gcd(m,\lambda_s) = 1$ for some $0 \leq s \leq r$. Then for all $\lambda\in \Z$ with $\gcd(m,\lambda) = 1$, the set $G(Q_s)$ is given by 
$$\left\{ mj+i\lambda\lambda_s \,\Big|\, 1 \leq i \leq m-1, \, \ceil*{\frac{-i\lambda\lambda_s}{m}} \leq j \leq  \sum_{k=0}^{r} \left\lceil \frac{i \lambda\lambda_k}{m}\right\rceil d_k-\left\lceil \frac{i \lambda\lambda_s}{m} \right\rceil - 1 \right\}.$$
In particular, let $\lambda$ be the inverse of $\lambda_s$ modulo $m$. Then
\begin{equation}\label{gap_set_simple_form}
    G(Q_s) = \left\{ mj+i ~\Big|~ 1 \leq i \leq m-1, \, 0 \leq j \leq  \sum_{k=0}^{r} \left\lceil \frac{i \lambda\lambda_k}{m}\right\rceil d_k -2 \right\}.
\end{equation}
\end{proposition}
\begin{proof}
Define the set 
	$$G = \left\{ mj+i\lambda\lambda_s ~\Big|~ 1 \leq i \leq m-1, \, \ceil*{\frac{-i\lambda\lambda_s}{m}} \leq j \leq  \sum_{k=0}^{r} \left\lceil \frac{i \lambda\lambda_k}{m}\right\rceil d_k -\left\lceil \frac{i \lambda\lambda_s}{m} \right\rceil - 1 \right\}.$$
Let $\lambda'\in \Z$ be the inverse of $\lambda_s$ modulo $m$. For $mj+i\lambda\lambda_s\in G$, by Lemma \ref{sum_mod}, we have
\begin{align*}
	\sum_{k=0}^{r}\floor*{\frac{-(mj+i\lambda\lambda_s)\lambda'\lambda_k}{m}}d_k = \sum_{k=0}^{r}\floor*{\frac{-i\lambda\lambda_s\lambda'\lambda_k}{m}}d_k = \sum_{k=0}^{r}\floor*{\frac{-i\lambda\lambda_k}{m}}d_k = -\sum_{k=0}^{r}\ceil*{\frac{i\lambda\lambda_k}{m}}d_k.
\end{align*}
It follows that
$$\sum_{k=0}^{r}\floor*{\frac{-(mj+i\lambda\lambda_s)\lambda'\lambda_k}{m}}d_k + \ceil*{\frac{mj+i\lambda\lambda_s}{m}} = j-\sum_{k=0}^{r}\ceil*{\frac{i\lambda\lambda_k}{m}}d_k+\ceil*{\frac{i\lambda\lambda_s}{m}}\leq -1.$$
By Proposition \ref{describe_gap_set}, we conclude that $mj+i\lambda\lambda_s\in G(Q_s)$. Hence $G\subseteq G(Q_s)$. 

Now let $mj_1+i_1\lambda\lambda_s, mj_2+i_2\lambda\lambda_s\in G$. If $mj_1+i_1\lambda\lambda_s = mj_2+i_2\lambda\lambda_s$, then $(i_1-i_2)\lambda\lambda_s \equiv 0\pmod m$.
Since $1\leq i_1,i_2\leq m-1$, we must have $i_1 = i_2$ and hence $j_1 = j_2$. By Lemma \ref{ceil_sum}, we obtain that
\begin{align*}
\#G&=\sum_{i=1}^{m-1}\left(\sum_{k=0}^{r}\ceil*{\frac{i\lambda\lambda_k}{m}}d_k-\ceil*{\frac{i\lambda\lambda_s}{m}}-\ceil*{\frac{-i\lambda\lambda_s}{m}}\right)\\
&= \sum_{i=1}^{m-1}\left(\sum_{k=0}^{r}\ceil*{\frac{i\lambda\lambda_k}{m}}d_k-\ceil*{\frac{i\lambda\lambda_s}{m}}+\floor*{\frac{i\lambda\lambda_s}{m}}\right) =\sum_{k=0}^{r}\sum_{i=1}^{m-1}\ceil*{\frac{i\lambda\lambda_k}{m}}d_k-(m-1)\\
&= \sum_{k=0}^{r}\frac{m(\lambda\lambda_k+1)-\lambda\lambda_k-\gcd(m,\lambda\lambda_k)}{2} \cdot d_k-(m-1)\\
&=\frac{2-2m+\sum_{k=0}^{r} \left(m-\gcd(m,\lambda_k)\right) \cdot d_k}{2} = g.
\end{align*}
This proves that $G(Q_s) = G$.

If $\lambda$ is the inverse of $\lambda_s$ modulo $m$, then $\lambda\lambda_s = 1+bm $ for some $b\in\Z$.
We have $\ceil*{\frac{-i\lambda\lambda_s}{m}} = -ib$ and $\ceil*{\frac{i\lambda\lambda_s}{m}} = ib+1$.
Then
$$G(Q_s) = \left\{ mj+ibm+i \mid 1 \leq i \leq m-1, \, -ib \leq j \leq  \sum_{k=0}^{r} \left\lceil \frac{i \lambda\lambda_k}{m} \right\rceil d_k -ib - 2 \right\}.$$
Letting $j' = j + ib$ yields the desired expression.
\end{proof}

A result similar to Proposition \ref{gap_set} also appears in the recent preprint \cite[Theorem 3.2]{cotterillGapSets2025}.
In contrast to \cite[Theorem 3.2]{cotterillGapSets2025}, we work with generalized Kummer extensions and permit the exponents $\lambda_i$ $(1\leq i\leq r)$ to be arbitrary integers, rather than restricting them to the range $1\leq \lambda_i \leq m-1$. In addition, we do not require $f(x)$ in the equation \eqref{kummer_extension} to be completely split over $K$. 
Moreover, by choosing an appropriate $\lambda$, we obtain the simplified expression in \eqref{gap_set_simple_form} for $G(Q_s)$.
Next we present some corollaries, which also appear on Kummer extensions in \cite{cotterillGapSets2025}.

\begin{corollary}\label{semigroup_equal}
    If $\lambda_i \equiv \lambda_j \pmod m$, $d_i = d_j =1$, and $\gcd(m,\lambda_i) = \gcd(m,\lambda_j) =1$ for some $0\leq i<j\leq r$, then $G(Q_i) = G(Q_j)$ and $H(Q_i) = H(Q_j)$.
\end{corollary}
\begin{proof}
    Let $\lambda$ be the inverse of $\lambda_i$ modulo $m$. Since $\lambda_i \equiv \lambda_j \pmod m$, then $\lambda$ is also the inverse of $\lambda_j$ modulo $m$. By Proposition \ref{gap_set}, we have $G(Q_i) = G(Q_j)$, and hence $H(Q_i) = H(Q_j)$.
\end{proof}
\begin{corollary}\label{semigroup}
Suppose that $d_s = \gcd(m,\lambda_s) = 1$ for some $0 \leq s \leq r$. Let $\lambda\in\Z$ be the inverse of $\lambda_s$ modulo $m$. Then
$$H(Q_s) = \left\langle m, m\left(\sum_{k=0}^{r}\ceil*{\frac{i \lambda\lambda_k}{m}}d_k-1\right) + i : 1\leq i\leq m-1 \right\rangle.$$
\end{corollary}
\begin{proof}
    By Proposition \ref{gap_set}, we have
    \begin{align*}
        H(Q_s) =&~ \N_0\setminus G(Q_s) \\
        =&~\left\{ mj+i ~\Big|~ 1 \leq i \leq m-1, \, j\geq \sum_{k=0}^{r} \left\lceil \frac{i \lambda\lambda_k}{m} \right\rceil d_k-1 \right\}\cup\{jm\mid j\geq 0\}\\
        =&~\left\langle m, m\left(\sum_{k=0}^{r}\ceil*{\frac{i \lambda\lambda_k}{m}}d_k-1\right) + i : 1\leq i\leq m-1 \right\rangle.
    \end{align*}
\end{proof}
\begin{corollary}\label{m_and_F}
Suppose that $d_s = \gcd(m,\lambda_s) = 1$ for some $0 \leq s \leq r$. Let $\lambda\in\Z$ be the inverse of $\lambda_s$ modulo $m$. Let $S = \left\{1\leq i\leq m-1 \mid \sum_{k=0}^{r}\ceil*{\frac{i\lambda\lambda_k}{m}}d_k-1= 0\right\}\cup\{m\}$.
Define $i_{\min} := \min S$, 
$$i_{\max}:=\max\left\{1\leq i\leq m-1\,\Big |\, \sum_{k=0}^{r}\ceil*{\frac{i\lambda\lambda_k}{m}}d_k-2=\max_{1\le j\le m-1} \sum_{k=0}^{r}\ceil*{\frac{j\lambda\lambda_k}{m}}d_k-2\,\right\},$$ and 
$$j_{\max} := \max_{1\leq i \leq m-1}  \sum_{k=0}^{r}\ceil*{\frac{i\lambda\lambda_k}{m}}d_k-2.$$
Then $m_{H(Q_s)} = i_{\min}$ and $F_{H(Q_s)} = mj_{\max} + i_{\max}$.
\end{corollary}
\begin{proof}
    By Corollary \ref{semigroup}, we have 
    $$ m_{H(Q_s)} = \min\left\{m, m\left(\sum_{k=0}^{r}\ceil*{\frac{i \lambda\lambda_k}{m}}-1\right)d_k + i : 1\leq i\leq m-1\right\}.$$
    Thus $m_{H(Q_s)} = \min S$.
    The second assertion follows directly from the definition of $F_{H(Q_s)}$ and Proposition \ref{gap_set}.
\end{proof}

\subsection{The symmetry of $H(Q_s)$}
 In this subsection, given a totally ramified place $Q_s$ of degree one, we provide a necessary and sufficient condition for $H(Q_s)$ to be symmetric.

\begin{theorem}\label{symmetric_thm}
	 Suppose that $d_s = \gcd(m,\lambda_s) = 1$ for some $0\leq s\leq r$. Then $H(Q_s)$ is symmetric if and only if there exists $1\leq u\leq m-1$ such that $$u\lambda_i\equiv\gcd(m,\lambda_i)\pmod m$$
       for all $i\in\{0,1,\dots,r\}\setminus\{s\}$.
\end{theorem}
   
\begin{proof}
    Suppose that $\lambda\lambda_s = am+1$ for some $\lambda\in \Z$ and $a\in\mathbb{Z}$. From Proposition \ref{gap_set}, we have
	\begin{align}
    G(Q_s) &= \left\{ mj+i ~\Big|~ 1 \leq i \leq m-1, \, 0 \leq j \leq  \sum_{k=0}^{r} \left\lceil \frac{i \lambda\lambda_k}{m} \right\rceil-2 \right\}\nonumber\\
              &= \left\{ mj+i ~\Big|~ 1 \leq i \leq m-1, \, 0 \leq j \leq  \sum_{k=0, k\neq s}^{r} \left\lceil \frac{i \lambda\lambda_k}{m} \right\rceil + ia-1 \right\} \label{one_gap}
   	\end{align}
    By the divisor \eqref{divisor3}, there exists a  Weil differential $\omega\in \Omega_F$ such that
    $$(\omega)_F =-\left(\frac{m}{\gcd(m,\lambda_0)}+1\right)\sum_{Q\in \P_{F},\,Q\mid P_0}Q  +\sum_{k=1}^{r}\left(\frac{m}{\gcd(m,\lambda_k)}-1\right)\sum_{Q\in \P_{F},\,Q\mid P_k}Q.$$
    For each $1\leq i\leq m-1$, using divisors \eqref{divisor2} and \eqref{divisor1}, we obtain
	$$ (y^{i\lambda})_F = \sum_{k=0,k\neq s}^{r}\frac{i\lambda\lambda_k}{\gcd(m,\lambda_k)}\sum_{Q\in \P_{F},\,Q\mid P_k}Q + (iam + i)Q_s,$$
    and 
	\begin{align*}
        \left(\left(z_s^{d_k}/z_k\right)^{\ceil*{\frac{i\lambda\lambda_k}{m}}}\right)_F = m\ceil*{\frac{i\lambda\lambda_k}{m}}d_k Q_s -\frac{m}{\gcd(m,\lambda_k)}\ceil*{\frac{i\lambda\lambda_k}{m}}\sum_{Q\in \P_{F},\,Q\mid P_k} Q 
    \end{align*}
	for each $ k\in \{0,1,\dots,r\}\setminus\{s\}$.
    Thus we get 
    \begin{align*}
     \left(y^{i\lambda}\left(\prod_{k=0}^{r}\left(z_s^{d_k}/z_k\right)^{\ceil*{\frac{i\lambda\lambda_k}{m}}}\right)z_s^{-2}\omega\right)_F = \left(m\left( \sum_{k=0,k\neq s}^{r}\ceil*{\frac{i \lambda\lambda_k}{m}}d_k + ia -1\right) + i-1 \right) Q_s\\
      + \sum_{k=0,k\neq s}^{r}\left(\frac{m}{\gcd(m,\lambda_k)}-1 + \frac{i\lambda\lambda_k}{\gcd(m,\lambda_k)} -\frac{m}{\gcd(m,\lambda_k)}\ceil*{\frac{i\lambda\lambda_k}{m}}\right)\sum_{Q\in \P_{F},\,Q\mid P_k} Q.
    \end{align*}
    For each $k\in \{0,1,\dots,r\}\setminus\{s\}$ and $1\leq i\leq m-1$, write $i\lambda\lambda_k = m\floor*{\frac{i\lambda\lambda_k}{m}}  + t_{ik}$ with $0\leq t_{ik}\leq m-1$.
	If $m\mid i\lambda\lambda_k$, then 
	\begin{equation}\label{sym_if1}
	    \frac{m}{\gcd(m,\lambda_k)}-1 + \frac{i\lambda\lambda_k}{\gcd(m,\lambda_k)} -\frac{m}{\gcd(m,\lambda_k)}\ceil*{\frac{i\lambda\lambda_k}{m}} = \frac{m}{\gcd(m,\lambda_k)} -1\geq 0.
	\end{equation}
    If $m\nmid i\lambda\lambda_k$, then $1\leq t_{ik}\leq m-1$ and 
    \begin{align}
        ~&\frac{m}{\gcd(m,\lambda_k)}-1 + \frac{i\lambda\lambda_k}{\gcd(m,\lambda_k)} -\frac{m}{\gcd(m,\lambda_k)}\ceil*{\frac{i\lambda\lambda_k}{m}} \nonumber\\
	  = ~&\frac{m}{\gcd(m,\lambda_k)}-1 + \frac{i\lambda\lambda_k}{\gcd(m,\lambda_k)} -\frac{m}{\gcd(m,\lambda_k)}\floor*{\frac{i\lambda\lambda_k}{m}} - \frac{m}{\gcd(m,\lambda_k)}\nonumber\\
      = ~&\frac{i\lambda\lambda_k}{\gcd(m,\lambda_k)} -\frac{m}{\gcd(m,\lambda_k)}\floor*{\frac{i\lambda\lambda_k}{m}}-1\nonumber\\
      = ~&\frac{i\lambda\lambda_k}{\gcd(m,\lambda_k)} -\frac{m}{\gcd(m,\lambda_k)}\left({\frac{i\lambda\lambda_k}{m}} - \frac{t_{ik}}{m}\right)-1 = \frac{t_{ik}}{\gcd(m,\lambda_k)}-1\geq 0. \label{sym_if2}
    \end{align}
    Hence 
	$$\frac{m}{\gcd(m,\lambda_k)}-1 + \frac{i\lambda\lambda_k}{\gcd(m,\lambda_k)} -\frac{m}{\gcd(m,\lambda_k)}\ceil*{\frac{i\lambda\lambda_k}{m}}\geq 0$$
	for all $k\in \{0,1,\dots,r\}\setminus\{s\}$ and $1\leq i\leq m-1$.
    For each $1\leq i\leq m-1$ satisfying $\sum_{k=0,k\neq s}^r\ceil*{\frac{i \lambda\lambda_k}{m}}d_k + ia -1\geq 0$, by Lemma \ref{differential_gap}, we obtain
	$$m\left(\sum_{k=0,k\neq s}^r\ceil*{\frac{i \lambda\lambda_k}{m}}d_k + ia -1\right)  + i \in G(Q_s).$$
	Moreover, the largest gap at $Q_s$ is obtained in this way according to \eqref{one_gap}.
    Note that every non-zero canonical divisor of $F$ has degree $2g-2$. Then by Lemma \ref{symmetric_to_canonical}, we conclude that $H(Q_s)$ is symmetric if and only if there exists $1\leq u\leq m-1$ such that 
	\begin{equation}\label{sym_condition}
	\frac{m}{\gcd(m,\lambda_k)}-1 + \frac{u\lambda\lambda_k}{\gcd(m,\lambda_k)} -\frac{m}{\gcd(m,\lambda_k)}\ceil*{\frac{u\lambda\lambda_k}{m}} = 0
	\end{equation}
	for all $k\in \{0,1,\dots,r\}\setminus\{s\}$.
    
	If $m\mid u\lambda\lambda_k$ for some $k\in \{0,1,\dots,r\}\setminus\{s\}$, then $t_{uk} = 0$. It follows from \eqref{sym_if1} that \eqref{sym_condition} holds if and only if $\gcd(m,\lambda_k) = m$.
	Thus \eqref{sym_condition} holds if and only if $t_{uk} \equiv 0 \pmod m$.
	
	If $m\nmid u\lambda\lambda_k$ for some $k\in \{0,1,\dots,r\}\setminus\{s\}$, it follows from \eqref{sym_if2} that \eqref{sym_condition} holds if and only if $t_{uk} \equiv \gcd(m,\lambda_k) \pmod m$.
	Since $t_{uk} \equiv u\lambda\lambda_k \pmod m$ for all $k\in \{0,1,\dots,r\}\setminus\{s\}$, we have
    \begin{align*}
    &~\eqref{sym_condition} \text{ holds for all } k\in \{0,1,\dots,r\}\setminus\{s\}\\
    \Leftrightarrow&~ t_{uk} \equiv \gcd(m,\lambda_k) \pmod m \text{ for all } k\in \{0,1,\dots,r\}\setminus\{s\}\\
    \Leftrightarrow&~ u\lambda \lambda_k\equiv \gcd(m,\lambda_k) \pmod m \text{ for all } k\in \{0,1,\dots,r\}\setminus\{s\}.\\
    \end{align*}
    Let $1\leq u'\leq m-1$ such that $u'\equiv u\lambda \pmod m$.
    Then $H(Q_s)$ is symmetric if and only if $u'\lambda_k \equiv \gcd(m,\lambda_k)\pmod m$ for all $k\in \{0,1,\dots,r\}\setminus\{s\}$.
\end{proof}
    In \cite[Theorems A and B]{mendozaKummerExtensions2023}, for the infinite place $Q_\infty$, Mendoza provided some conditions for the Weierstrass semigroup $H(Q_\infty)$ to be symmetric. We generalize these results to arbitrary totally ramified places of degree one on the generalized Kummer extension in the following two corollaries.
\begin{corollary}\label{condition1}
    For each $0\leq i\leq r$, write $\lambda_i = mb_i + \bar{\lambda}_i$, where $b_i\in \Z$ and $1\leq \bar{\lambda}_i\leq m$. Suppose that $d_s = \gcd(m,\lambda_{s}) = 1$ for some $0\leq s\leq r$.
    Then 
    $$F_{H(Q_s)} = m\left(\sum_{i=0}^{r} d_i-2\right)-\sum_{i\neq s,\,i=0}^{r}\bar{\lambda}_i d_i \text{ and } H(Q_s) \text{ is symmetric},$$
    if and only if $\bar{\lambda}_i\mid m$ for all $i\in \{0,1,\dots,r\}\setminus\{s\}$.
\end{corollary}
\begin{proof}
    If $\bar{\lambda}_i\mid m$ for all $i\in\{0,1,\dots,r\}\setminus\{s\}$, then $\lambda_i \equiv \gcd(m,\lambda_i) \pmod m$ for all $i\in \{0,1,\dots,r\}\setminus\{s\}$. It follows from Theorem \ref{symmetric_thm} that $H(Q_s)$ is symmetric and 
    \begin{align*}
	F_{H(Q_s)} = 2g-1 &= m\left(\sum_{i=0}^{r}d_i-2\right)+2-\sum_{i=0}^{r}\gcd(m,\lambda_i)\cdot d_i-1 \\
	&= m\left(\sum_{i=0}^{r}d_i-2\right)-\sum_{i\neq s,\,i=0}^{r}\bar{\lambda}_i d_i.
	\end{align*}

    Conversely, assume that $F_{H(Q_s)} = m\left(\sum_{i=0}^{r}d_i-2\right)-\sum_{i\neq s,\,i=0}^{r}\bar{\lambda}_i d_i$ and $H(Q_s)$ is symmetric. Then
    \begin{align*}
	F_{H(Q_s)} &= m\left(\sum_{i=0}^{r}d_i-2\right)-\sum_{i\neq s,\,i=0}^{r}\bar{\lambda}_i d_i \\
	&= 2g-1 = m\left(\sum_{i=0}^{r}d_i-2\right)- \sum_{i\neq s,\,i=0}^{r}\gcd(m,\lambda_i)\cdot d_i.
	\end{align*}
    This implies that $\bar{\lambda}_i\mid m$ for all $i\in \{0,1,\dots,r\}\setminus\{s\}$.
\end{proof}

\begin{corollary}\label{condition2}
     Let $0\leq s\leq r$. Suppose that $\gcd(m,\lambda_i) = 1$ for all $0\leq i\leq r$ and $d_s = 1$.
    Then $H(Q_s)$ is symmetric if and only if $\lambda_i \equiv \lambda_j \pmod m$ for all $i,j\in \{0,1,\dots,r\}\setminus\{s\}$.
\end{corollary}
\begin{proof}
    Since $\gcd(m,\lambda_i) = 1$ for all $0\leq i\leq r$, it follows from Theorem \ref{symmetric_thm} that $H(Q_s)$ is symmetric if and only if there exists $1\leq u\leq m-1$ such that $u\lambda_i \equiv 1\pmod m$ for all $i\in \{0,1,\dots,r\}\setminus\{s\}$. Let $u'$ be the inverse of $u$ module $m$. Then $u\lambda_i \equiv 1\pmod m$ for all $i\in \{0,1,\dots,r\}\setminus\{s\}$ is equivalent to $\lambda_i \equiv u'\pmod m$ for all $i\in \{0,1,\dots,r\}\setminus\{s\}$. The proof is completed.
\end{proof}

    Besides the conditions in the above corollaries, there may exist other situations in which the Weierstrass semigroup is symmetric.
	For instance, although the conditions in the following example do not meet those required by the above two corollaries, the Weierstrass semigroup $H(Q_\infty)$ is still symmetric
\begin{example}
    Let $F = \mathbb{F}_{11}(x,y)/\mathbb{F}_{11}(x)$ be a function field defined by $y^8 = x^3(x-1)^6$. Take $m = 8$, $\lambda_1 = 3$, $\lambda_2 = 6$ and $\lambda_ 0 = -9$.
	Then $\gcd(m,\lambda_1) = 1$ and $\gcd(m,\lambda_2) = 2$, so $3\lambda_1 \equiv \gcd(m,\lambda_1) \pmod m$ and $3\lambda_2\equiv \gcd(m,\lambda_2)\pmod m$. Hence $H(Q_\infty)$ is symmetric. 
\end{example}

\section{The minimal generating sets of Weierstrass semigroups on generalized Kummer extensions}\label{section4}
In this section, we investigate the minimal generating sets of the Weierstrass semigroups at several totally ramified places of degree one on the generalized Kummer extension $F = K(x,y)/K(x)$ defined by the equation \eqref{kummer_extension}.
Let $\bar{r}:=\min\{r+1,|K|\}-1$ and let $\{\ell_0,\ell_1,\dots,\ell_r\}$ be a permutation of $\{0,1,\dots,r\}$. We present our main results below.
\begin{theorem}\label{thm_minimal_set}
	Let $1\leq s\leq \bar{r}$ and $\la\in \Z$ with $\gcd(m,\la) = 1$. Suppose that $d_{\ell_k} = \gcd(m,\la_{\ell_k}) = 1$ for all $0\leq k\leq s$.
	Then
	\begin{align*}
		\Gamma(Q_{\ell_0},\dots, Q_{\ell_s}) = \Bigg\{(mj_0+i\la\la_{\ell_0},\dots, mj_s+i\la\la_{\ell_s})\in \mathbb{N}^{s+1} ~\Big| ~1\leq i \leq m-1 , \\
		 ~j_k\geq \ceil*{\frac{-i\la\la_{\ell_k}}{m}}\text{ for }0\leq k\leq s,~ \sum_{k=0}^{s}j_k=\sum_{k=0}^r\ceil*{\frac{i\la\la_k}{m}}d_k- \sum_{k=0}^{s}\ceil*{\frac{i\la\la_{\ell_k}}{m}}\Bigg\}.
		\end{align*}
	Moreover, for each $(mj_0+i\la\la_{\ell_0},\dots, mj_s+i\la\la_{\ell_s})\in \Gamma(Q_{\ell_0},\dots, Q_{\ell_s})$, 
	$$\left( \frac{\prod_{k=s+1}^r z_{\ell_k}^{\ceil*{\frac{i\la\la_{\ell_k}}{m}}}}{y^{i\la} \prod_{k=0}^{s} z_{\ell_k}^{j_k} }\right)_{\infty} = \sum_{k=0}^{s}(mj_k+i\la\la_{\ell_k})Q_{\ell_k}.$$
	In addition, the set $\Gamma(Q_{\ell_0},\dots, Q_{\ell_s}) = \varnothing$ if and only if $\max_{1\leq i\leq m}\sum_{k=0}^r\ceil* {\frac{i\la_{k}}{m}}d_k \leq s\leq \bar{r}$.
\end{theorem}
A result similar to Theorem \ref{thm_minimal_set} also appears in the recent work \cite[Corollary 3.5]{castellanosGeneralizedWeierstrass2026}.
This work consider the Kummer extension defined by $y^m  = f(x)$ with $f(x)\in K[x]$.
Let $t_k(i) = i\lambda\lambda_{\ell_k} - \floor*{\frac{i\lambda\lambda_{\ell_k}}{m}}m$ and $j_k' = j_0 + \floor*{\frac{i\lambda\lambda_{\ell_k}}{m}}$ for $0\leq k\leq s$.
Suppose that $d_k = 1$ for all $1\leq k\leq r$, then 
\begin{align*}
	\sum_{k=0}^s j_k' = \sum_{k=0}^r\ceil*{\frac{i\la\la_k}{m}} - \sum_{k=0}^{s}\ceil*{\frac{i\la\la_{\ell_k}}{m}}  +  \sum_{k=0}^{s}\floor*{\frac{i\la\la_{\ell_k}}{m}} = \sum_{k=0}^r\ceil*{\frac{i\la_k}{m}} - (s+1).
\end{align*}
This implies that our result coincides with \cite[Corollary 3.5]{castellanosGeneralizedWeierstrass2026} when $f(x)\in K[x]$.

In contrast to \cite[Corollary 3.5]{castellanosGeneralizedWeierstrass2026}, we allow the exponents $\lambda_i$ $(1\leq i\leq r)$ to be arbitrary integers, rather than restricting them to positive integer. That is, $f(x)$ may lie in $K(x)$, rather than being restricted to $K[x]$.
In addition, we do not require $f(x)$ to be completely split over $K$.
The function field we considered is more general, as $K$ is not required to contain a primitive $m$-th root of unity.
Moreover, we provide explicit functions whose pole divisors have coefficients lying in $\Gamma(Q_{\ell_0},\dots,Q_{\ell_s})$, along with a necessary and sufficient condition for $\Gamma(Q_{\ell_0},\dots,Q_{\ell_s})$ to be nonempty.
By selecting a suitable $\lambda$, we are able to obtain a simplified expression for $\Gamma(Q_{\ell_0},\dots,Q_{\ell_s})$ when some of the $\lambda_{\ell_k}$ are congruent modulo $m$.
\begin{corollary}\label{Gamma_some_equal}
    Let $1\leq s\leq \bar{r}$ and $0\leq t\leq s$. Suppose that $d_{\ell_k} = \gcd(m,\la_{\ell_k}) = 1$ for all $0\leq k\leq s$ and $\lambda_{\ell_0} \equiv \dots \equiv\lambda_{\ell_t} \pmod m$.
    Let $\lambda\in\mathbb{Z}$ be the inverse of $\lambda_{\ell_0}$ modulo $m$.
	Then for $1\leq s\leq \max_{1\leq i\leq m-1}\sum_{k=0}^r\ceil* {\frac{i\la_{k}}{m}}d_k-1$,
	\begin{align*}
		\Gamma(Q_{\ell_0},\dots, Q_{\ell_s}) = \Bigg\{(mj_0+i,\dots, mj_t+i,mj_{t+1} + i\lambda\lambda_{\ell_{t+1}},\dots, mj_s + i\lambda\lambda_{\ell_s})\in \mathbb{N}^{s+1} ~\Big| \\
        1\leq i \leq m-1,\, j_k\geq 0 \text{ for } 0\leq k\leq t,\, j_k\geq \ceil*{\frac{-i\la\la_{\ell_k}}{m}}\text{ for }t+1\leq k\leq s,\\
        \sum_{k=0}^{s}j_k=\sum_{k=0}^r\ceil*{\frac{i\la\la_k}{m}}d_k- \sum_{k=t+1}^{s}\ceil*{\frac{i\la\la_{\ell_k}}{m}} -t-1 \Bigg\}.
		\end{align*}
	For each $(mj_0+i,\dots, mj_t+i,mj_{t+1} + i\lambda\lambda_{\ell_{t+1}},\dots, mj_s + i\lambda\lambda_{\ell_s})\in \Gamma(Q_{\ell_0},\dots, Q_{\ell_s})$,
	$$\left( \frac{\prod_{k=s+1}^r z_{\ell_k}^{\ceil*{\frac{i\la\la_{\ell_k}}{m}}}}{y^{i\la} \prod_{k=0}^{s} z_{\ell_k}^{j_k} }\right)_{\infty} = \sum_{k=0}^{t}(mj_k+i)Q_{\ell_k} + \sum_{k=t+1}^{s}(mj_k+i\lambda\lambda_{\ell_k})Q_{\ell_k}.$$
\end{corollary}
\begin{proof}
    For $1\leq s\leq \max_{1\leq i\leq m-1}\sum_{k=0}^r\ceil* {\frac{i\la_{k}}{m}}d_k-1$, it follows from Theorem \ref{thm_minimal_set} that 
    \begin{align*}
		\Gamma(Q_{\ell_0},\dots, Q_{\ell_s}) = \Bigg\{(mj'_0+i\la\la_{\ell_0},\dots, mj'_s+i\la\la_{\ell_s})\in \mathbb{N}^{s+1} ~\Big| ~1\leq i \leq m-1 , \\
		 ~j'_k\geq \ceil*{\frac{-i\la\la_{\ell_k}}{m}}\text{ for }0\leq k\leq s,~ \sum_{k=0}^{s}j'_k=\sum_{k=0}^r\ceil*{\frac{i\la\la_k}{m}}d_k- \sum_{k=0}^{s}\ceil*{\frac{i\la\la_{\ell_k}}{m}}\Bigg\}.
		\end{align*}
    For each $0\leq k\leq t$, let $b_k\in \Z$ such that $\la\la_{\ell_k} = mb_k + 1$. Then for each $1\leq i\leq m-1$, we have $\ceil*{\frac{-i\la\la_{\ell_k}}{m}} = -ib_k$ and $\ceil*{\frac{i\la\la_{\ell_k}}{m}} = ib_k+1$ .
	Let $j_k = j'_k+ib_k$ for $0\leq k\leq t$ and $j_k = j_k'$ for $t+1\leq k\leq s$. Then we obtain the desired expression.
\end{proof}

We will prove Theorem \ref{thm_minimal_set} by induction on $s$. First, we show that Theorem \ref{thm_minimal_set} holds for $s=1$.
\begin{proposition}\label{two_places_Gamma}
	Suppose that $d_{\ell_0} = \gcd(m,\la_{\ell_0}) =1$ and $d_{\ell_1} = \gcd(m, \la_{\ell_1})=1$.
	Then for all $\la\in \Z$ with $\gcd(m,\la) = 1$,
	\begin{align*}
		\Gamma(Q_{\ell_0}, Q_{\ell_1})=\Bigg\{&(mj_0 + i\la\la_{\ell_0}, mj_1 + i\la\la_{\ell_1})\in \N^2 \mid 1\leq i \leq m-1, \, j_0\geq \ceil*{\frac{-i\la\lambda_{\ell_0}}{m}},\\
		 &~j_1\geq \ceil*{\frac{-i\la\lambda_{\ell_1}}{m}}, \quad j_0+j_1=\sum_{k=0}^{r}\ceil*{\frac{i\la\la_k}{m}}d_k-\ceil*{\frac{i\la\la_{\ell_0}}{m}}-\ceil*{\frac{i\la\la_{\ell_1}}{m}} \Bigg\}.
	\end{align*}
\end{proposition}
\begin{proof}
	Define the set 
\begin{align*}
	\Gamma=\Bigg\{&(mj_0 + i\la\la_{\ell_0}, mj_1 + i\la\la_{\ell_1})\in \N^2 \mid 1\leq i \leq m-1, \, j_0\geq \ceil*{\frac{-i\la\lambda_{\ell_0}}{m}},\\
		 &~j_1\geq \ceil*{\frac{-i\la\lambda_{\ell_1}}{m}}, \quad j_0+j_1=\sum_{k=0}^{r}\ceil*{\frac{i\la\la_k}{m}}d_k-\ceil*{\frac{i\la\la_{\ell_0}}{m}}-\ceil*{\frac{i\la\la_{\ell_1}}{m}} \Bigg\}.
\end{align*}

Given a tuple $(mj_0 + i\la\la_{\ell_0}, mj_1 + i\la\la_{\ell_1})\in\Gamma$, by the divisors \eqref{divisor1} and \eqref{divisor2}, we obtain the following principal divisor:
\begin{align*}
	&\left(y^{-i\la}z_{\ell_0}^{-j_0}z_{\ell_1}^{-j_1}\prod_{k=2}^r z_{\ell_k}^{\ceil*{\frac{i\la\la_{\ell_k}}{m}}}\right)_F\\
 =&~-i\la\la_{\ell_0}Q_{\ell_0}-i\la\la_{\ell_1}Q_{\ell_1}+\sum_{k=2}^r\frac{-i\la\la_{\ell_k}}{\gcd(m,\la_{\ell_k})}\sum_ {Q\in \P_{F}, \, Q|P_{\ell_k}}Q\\
 &~-mj_0Q_{\ell_0}+\frac{mj_0}{\gcd(m,\la_0)}\sum_ {Q\in \P_{F}, \, Q|P_{0}}Q - mj_1Q_{\ell_1}+\frac{mj_1}{\gcd(m,\la_0)}\sum_ {Q\in \P_{F}, \, Q|P_{0}}Q\\
 &~+\sum_{k=2}^r\frac{m\ceil*{\frac{i\la\la_{\ell_k}}{m}}}{\gcd(m,\la_{\ell_k})}\sum_ {Q\in \P_{F}, \, Q|P_{\ell_k}}Q - \sum_{k=2}^r\frac{m\ceil*{\frac{i\la\la_{\ell_k}}{m}}d_{\ell_k}}{\gcd(m,\la_0)}\sum_ {Q\in \P_{F}, \, Q|P_{0}}Q\\
 =&~\sum_{k=2}^r\frac{m\ceil*{\frac{i\la\la_{\ell_k}}{m}}-i\la\la_{\ell_k}}{\gcd(m,\la_{\ell_k})}\sum_ {Q\in \P_{F}, \, Q|P_{\ell_k}}Q-(mj_0 + i\la\la_{\ell_0})Q_{\ell_0}-(mj_1 + i\la\la_{\ell_1})Q_{\ell_1}.
\end{align*}
Since $m\ceil*{\frac{i\la\la_{\ell_k}}{m}}-i\la\la_{\ell_k} \geq 0$ for all $2\leq k\leq r$, we conclude that $(mj_0+i\la\la_{\ell_0},mj_1+i\la\la_{\ell_1})\in H(Q_{\ell_0},Q_{\ell_1})$. Therefore $\Gamma\subseteq H(Q_{\ell_0},Q_{\ell_1})$.

On the other hand, since $j_0\geq \ceil*{\frac{-i\la\lambda_{\ell_0}}{m}}$, $j_1\geq \ceil*{\frac{-i\la\lambda_{\ell_1}}{m}}$, and $$j_0+j_1=\sum_{k=0}^{r}\ceil*{\frac{i\la\la_k}{m}}d_k-\ceil*{\frac{i\la\la_{\ell_0}}{m}}-\ceil*{\frac{i\la\la_{\ell_1}}{m}},$$ we have
\begin{align*}
	\ceil*{\frac{-i\la\lambda_{\ell_0}}{m}}\leq j_0&\leq \sum_{k=0}^{r}\ceil*{\frac{i\la\la_k}{m}}d_k-\ceil*{\frac{i\la\la_{\ell_0}}{m}}-\ceil*{\frac{i\la\la_{\ell_1}}{m}}-\ceil*{\frac{-i\la\lambda_{\ell_1}}{m}}\\
	& = \sum_{k=0}^{r}\ceil*{\frac{i\la\la_k}{m}}d_k-\ceil*{\frac{i\la\la_{\ell_0}}{m}}-\ceil*{\frac{i\la\la_{\ell_1}}{m}}+\floor*{\frac{i\la\lambda_{\ell_1}}{m}}\\
	& = \sum_{k=0}^{r}\ceil*{\frac{i\la\la_k}{m}}d_k-\ceil*{\frac{i\la\la_{\ell_0}}{m}} -1.
\end{align*}
Similarly, we also have $\ceil*{\frac{-i\la\lambda_{\ell_1}}{m}}\leq j_1\leq \sum_{k=0}^{r}\ceil*{\frac{i\la\la_k}{m}}d_k-\ceil*{\frac{i\la\la_{\ell_1}}{m}} -1$. This shows that $\Gamma\subseteq G(Q_{\ell_0})\times G(Q_{\ell_1})$. We conclude that $\Gamma\subseteq (G(Q_{\ell_0})\times G(Q_{\ell_1})\cap H(Q_{\ell_0},Q_{\ell_1})$.

Moreover, the set $\Gamma$ can be seen as the graph of the bijection $\theta:~ G(Q_{\ell_0})\rightarrow G(Q_{\ell_1})$ given by $\theta(mj_0+i\la\la_{\ell_0}) = mj_1 + i\la\la_{\ell_1}$,
which defines a permutation $\tau$ of the set $\{1,\dots,g\}$. By Lemma \ref{minimal_generating_two}, we obtain that $\Gamma = \Gamma(Q_{\ell_0},Q_{\ell_1})$.
\end{proof}

To begin proving that Theorem \ref{thm_minimal_set} holds for $s\geq 2$, we first state a lemma. By Lemma \ref{lemma5} and Proposition \ref{gap_set}, the following lemma is well-defined.

\begin{lemma}\label{i_is_same}
	Suppose that $1\leq s\leq \bar{r}$ and  $d_{\ell_k} = \gcd(m,\la_{\ell_k}) = 1$ for all $0\leq k\leq s$.
	Let $\n = (n_0,\dots,n_s)\in \Gamma(Q_{\ell_0},\dots,Q_{\ell_s})$ and let $\la\in \Z$ with $\gcd(m,\la)=1$.
	Suppose that $$\n = (mj_0+i_0\la\la_{\ell_0},\dots,mj_s+i_s\la\la_{\ell_s}),$$
	where $ 1 \leq i_k \leq m-1 $ and $ j_k \geq \ceil*{\frac{-i_k\la\la_{\ell_k}}{m}} $ for $0 \leq k \leq s$.
	Then $i_0 = i_k$ for all $1\leq k\leq s$.
\end{lemma}
\begin{proof}
	Suppose that $i_0 \neq i_k$ for some $1\leq k\leq s$. Without loss of generality, we assume that $k=1$.
	By Lemma \ref{discrepancy}, we obtain
	\begin{equation}\label{dim_1}
		\ell\left(\sum_{k=0}^{s}n_kQ_{\ell_k}\right) = \ell\left(\sum_{k=0}^{s}n_kQ_{\ell_k} - Q_{\ell_0}\right)+1,
	\end{equation}
	\begin{equation}\label{dim_2}
		\ell\left(\sum_{k=0}^{s}n_kQ_{\ell_k}\right) = \ell\left(\sum_{k=0}^{s}n_kQ_{\ell_k}- Q_{\ell_1}\right)+1,
	\end{equation}
	\begin{equation}\label{dim_3}
		\ell\left(\sum_{k=0}^{s}n_kQ_{\ell_k}\right) = \ell\left(\sum_{k=0}^{s}n_kQ_{\ell_k}-Q_{\ell_0}-Q_{\ell_1}\right)+1.
	\end{equation}
	
	Let $1\leq t_0\leq m-1$. Note that $mj_0 + i_0\lambda\lambda_{\ell_0} + t_0\lambda_{\ell_0}\equiv 0\pmod m$ if and only if $t_0\equiv -i_0\la\pmod m$.
	We have $\floor*{\frac{n_0+t_0\la_{\ell_0}}{m}} = \floor*{\frac{n_0-1+t_0\la_{\ell_0}}{m}}+1$ if and only if $t_0\equiv -i_0\la\pmod m$.
	Then by \eqref{dim_1} and Corollary \ref{D_dimension}, we have
	\begin{align}\label{dim_4}
		\sum_{k=0}^{s}\floor*{\frac{n_k+t_0\la_{\ell_k}}{m}}+\sum_{k=s+1}^{r}\floor*{\frac{t_0\la_{\ell_k}}{m}} d_{\ell_k}\geq 0.
	\end{align}
	Let $1\leq t_1\leq m-1$. Note that $mj_1 + i_1\lambda\lambda_{\ell_1} + t_1\lambda_{\ell_1}\equiv 0\pmod m$ if and only if $t_1\equiv -i_1\la\pmod m$.
	We have $\floor*{\frac{n_1+t_1\la_{\ell_1}}{m}} = \floor*{\frac{n_1-1+t_1\la_{\ell_1}}{m}}+1$ if and only if $t_1\equiv -i_1\la\pmod m$.
	Moreover, we have $t_1\neq t_0$ since $i_1\neq i_0$. 
	Then by \eqref{dim_3}, \eqref{dim_4} and Corollary \ref{D_dimension}, we have
	\begin{align*}
		\sum_{k=0}^{s}\floor*{\frac{n_k+t_1\la_{\ell_k}}{m}}+\sum_{k=s+1}^{r}\floor*{\frac{t_1\la_{\ell_k}}{m}}d_{\ell_k}\leq -1
	\end{align*}
	On the other hand, by \eqref{dim_2} and Corollary \ref{D_dimension}, we have
	\begin{align*}
		\sum_{k=0}^{s}\floor*{\frac{n_k+t_1\la_{\ell_k}}{m}}+\sum_{k=s+1}^{r}\floor*{\frac{t_1\la_{\ell_k}}{m}}d_{\ell_k}\geq 0,
	\end{align*}
	which is a contradiction. Thus $i_0 = i_k$ for all $1\leq k\leq s$.
\end{proof}

\begin{definition}
	Let $1\leq s\leq \bar{r}$ and $\la\in \Z$ with $\gcd(m,\la) = 1$. Suppose that $d_{\ell_k} = \gcd(m,\la_{\ell_k}) = 1$ for all $0\leq k\leq s$.
	Define the set 
	\begin{align*}
		\Gamma_\la(Q_{\ell_0},\dots, Q_{\ell_s}) := \Bigg\{
			\mathbf{u}_{\j,i\la} = (mj_0+i\la\la_{\ell_0},\dots, mj_s+i\la\la_{\ell_s})\in \mathbb{N}^{s+1} ~\Big|
		~1\leq i \leq m-1 , \\
		 ~j_k\geq \ceil*{\frac{-i\la\la_{\ell_k}}{m}}\text{ for }0\leq k\leq s,~ \sum_{k=0}^{s}j_k=\sum_{k=0}^r\ceil*{\frac{i\la\la_k}{m}}d_k- \sum_{k=0}^{s}\ceil*{\frac{i\la\la_{\ell_k}}{m}}\Bigg\}.
		\end{align*}
\end{definition}
Our goal is to show that $\Gamma_\la(Q_{\ell_0},\dots, Q_{\ell_s}) = \Gamma(Q_{\ell_0},\dots, Q_{\ell_s})$ for all $\la\in \Z$ with $\gcd(m,\la) = 1$. We first show that $\Gamma_\lambda(Q_{\ell_0},\dots,Q_{\ell_s})\subseteq\Gamma(Q_{\ell_0},\dots, Q_{\ell_s})$.

\begin{proposition}\label{S_subseteq_Gamma}
	Let $1\leq s\leq \bar{r}$ and $\la\in \Z$ with $\gcd(m,\la) = 1$. Suppose that $d_{\ell_k} = \gcd(m,\la_{\ell_k}) = 1$ for all $0\leq k\leq s$.
	Then
	$$\Gamma_\lambda(Q_{\ell_0},\dots, Q_{\ell_s})\subseteq\Gamma(Q_{\ell_0},\dots, Q_{\ell_s}).$$
	Moreover, for each $\mathbf{u}_{\j,i\la} = (mj_0+i\la\la_{\ell_0},\dots, mj_s+i\la\la_{\ell_s})\in \Gamma_\lambda(Q_{\ell_0},\dots, Q_{\ell_s})$, 
	$$\left( \frac{\prod_{k=s+1}^r z_{\ell_k}^{\ceil*{\frac{i\la\la_{\ell_k}}{m}}}}{y^{i\la} \prod_{k=0}^{s} z_{\ell_k}^{j_k} }\right)_{\infty} = \sum_{k=0}^{s}(mj_k+i\la\la_{\ell_k})Q_{\ell_k}.$$
\end{proposition}
\begin{proof}
 	We will prove the proposition by induction on $s$.
	By Proposition \ref{two_places_Gamma}, we have $\Gamma_\la(Q_{\ell_0},Q_{\ell_1})=\Gamma(Q_{\ell_0},Q_{\ell_1})$,
	which establishes the case where $ s=1 $. Assume that $\Gamma_\la(Q_{\ell_0},\dots,Q_{\ell_{k}}) \subseteq \Gamma(Q_{\ell_0},\dots,Q_{\ell_{k}})$ holds for all $ 1 \leq k \leq s-1 $, where $s\geq 2$.
	Let $\mathbf{u}_{\j,i\la} = (mj_0+i\la\la_{\ell_0},\dots, mj_s+i\la\la_{\ell_s})\in \Gamma_\lambda(Q_{\ell_0},\dots, Q_{\ell_s})$. Then by the divisors \eqref{divisor1} and \eqref{divisor2}, we obtain
	\begin{align*}
		&\left( \left(y^{i\la} \prod_{k=0}^{s} z_{\ell_k}^{j_k}\right)^{-1}\prod_{k=s+1}^r z_{\ell_k}^{\ceil*{\frac{i\la\la_{\ell_k}}{m}}}\right)_F\\
		=&~-\sum_{k=0}^{s}i\la\la_{\ell_k}Q_{\ell_k}+\sum_{k=s+1}^r\frac{-i\la\la_{\ell_k}}{\gcd(m,\la_{\ell_k})}\sum_ {Q\in \P_{F}, \, Q|P_{\ell_k}}Q\\
 		&~-\sum_{k=0}^{s}mj_kQ_{\ell_k}+\sum_{k=0}^{s}\frac{mj_k}{\gcd(m,\la_0)}\sum_ {Q\in \P_{F}, \, Q|P_{0}}Q\\
 		&~+\sum_{k=s+1}^r\frac{m\ceil*{\frac{i\la\la_{\ell_k}}{m}}}{\gcd(m,\la_{\ell_k})}\sum_ {Q\in \P_{F}, \, Q|P_{\ell_k}}Q - \sum_{k=s+1}^r\frac{m\ceil*{\frac{i\la\la_{\ell_k}}{m}}d_{\ell_k}}{\gcd(m,\la_0)}\sum_ {Q\in \P_{F}, \, Q|P_{0}}Q\\
 		=&~\sum_{k=s+1}^r\frac{m\ceil*{\frac{i\la\la_{\ell_k}}{m}}-i\la\la_{\ell_k}}{\gcd(m,\la_{\ell_k})}\sum_ {Q\in \P_{F}, \, Q|P_{\ell_k}}Q-\sum_{k=0}^{s}(mj_k+i\la\la_{\ell_k})Q_{\ell_k}.
	\end{align*}
	Since $m\ceil*{\frac{i\la\la_{\ell_k}}{m}}-i\la\la_{\ell_k}\geq 0$ for $s+1\leq k\leq r$, we have $\mathbf{u}_{\mathbf{j},i\la} \in H(Q_{\ell_0},\dots,Q_{\ell_s})$.

	In order to show that $\mathbf{u}_{\mathbf{j},i\la} \in \Gamma(Q_{\ell_0},\dots,Q_{\ell_s})$, it suffices to prove that $\mathbf{u}_{\mathbf{j},i\la}$ is minimal in $\{\mathbf{n} = (n_0,\dots,n_s)\in H(Q_{\ell_0},\dots,Q_{\ell_s})~|~ n_0=mj_0+i\la\la_{\ell_0}\}$.
	Suppose that $\mathbf{u}_{\mathbf{j},i\la}$ is not minimal in $\{\mathbf{n} = (n_0,\dots,n_s)\in H(Q_{\ell_0},\dots,Q_{\ell_s})~|~ n_0=mj_0+i\la\la_{\ell_0}\}$.
	Then there exists $ \mathbf{u}=(u_0,\dots,u_s)\in H(Q_{\ell_0},\dots,Q_{\ell_s})$ with $u_0=mj_0+i\la\la_{\ell_0}$, and $\mathbf{u}\prec\mathbf{u}_{\mathbf{j},i\la}$.
	Let $h \in F$ such that $(h)_{\infty} = \sum_{k=0}^{s}u_kQ_{\ell_k}$. Note that $\mathbf{u} \prec \mathbf{u}_{\mathbf{j},i\la}$ gives $u_k < mj_k+i\la\la_{\ell_k}$ for some $1\leq k \leq s$.
	Without loss of generality, we may assume that $ u_{1}< mj_{1}+i\la\la_{\ell_{1}}$.

	We take $v_{s-1} = mj_{s}+m\ceil*{\frac{i\la\la_{\ell_s}}{m}} + u_{s-1}$ and $v_k = u_k$ for $0\leq k\leq s-2$. Let $\mathbf{v} = (v_0,\dots,v_{s-1})$.
	Since $-u_{s}+m\left(j_{s}+\ceil*{\frac{i\la\la_{\ell_s}}{m}}\right) \geq m\ceil*{\frac{i\la\la_{\ell_s}}{m}}-i\la\la_{\ell_s}>0$, we have
	$$\left( h z_{\ell_s}^{j_{s}+\ceil*{\frac{i\la\la_{\ell_s}}{m}}} z_{\ell_{s-1}}^{-j_{s}-\ceil*{\frac{i\la\la_{\ell_s}}{m}}} \right)_{\infty} = \sum_{k=0}^{s-1} v_k Q_{\ell_k},$$
	which implies that $\mathbf{v} \in H(Q_{\ell_0},\dots,Q_{\ell_{s-1}})$.
	Let $w_{s-1} = m\left(j_{s-1}+j_{s}+\ceil*{\frac{i\la\la_{\ell_s}}{m}}\right)+i\la\la_{\ell_{s-1}}$ and 
	$$\mathbf{w} =\left(mj_0+i\la\la_{\ell_0},\dots,mj_{s-2}+i\la\la_{\ell_{s-2}},w_{s-1}\right).$$
    Then we obtain $\mathbf{w}\in \Gamma_\la(Q_{\ell_0},\dots,Q_{\ell_{s-1}})$ and $v_1<w_1$.
	It follows from the induction hypothesis that $\Gamma_\la(Q_{\ell_0},\dots,Q_{\ell_{s-1}})\subseteq\Gamma(Q_{\ell_0},\dots,Q_{\ell_{s-1}})$.
    Thus $\mathbf{w} \in \Gamma(Q_{\ell_0},\dots,Q_{\ell_{s-1}})$.
    Then by Proposition \ref{minimal_for_all}, we get that $\mathbf{w}$ is minimal in the set
	$\{\mathbf{n} = (n_0,\dots,n_{s-1})\in H(Q_{\ell_0},\dots,Q_{\ell_{s-1}})~|~n_0=mj_0+i\la\la_{\ell_0}\}$.
	Now we have
	\begin{align*}
	\mathbf{v} \in \{\mathbf{n} = (n_0,\dots,n_{s-1})\in H(Q_{\ell_0},\dots,Q_{\ell_{s-1}})~|~n_0=mj_0+i\la\la_{\ell_0}\} \text{ and } \mathbf{v} \prec \mathbf{w},
	\end{align*}
	which is a contradiction to the minimality of $\mathbf{w}$. It follows that $ \mathbf{u}_{\mathbf{j},i\la} $ is minimal in $\{\mathbf{n} = (n_0,\dots,n_s)\in H(Q_{\ell_0},\dots,Q_{\ell_s})~|~ n_0=mj_0+i\la\la_{\ell_0}\}$.
	Therefore, we have $u_{\mathbf{j},i\la} = (mj_0+i\la\la_{\ell_0},\dots, mj_s+i\la\la_{\ell_s})\in \Gamma(Q_{\ell_0},\dots,Q_{\ell_s})$ and 
	$$\left( \frac{\prod_{k=s+1}^r z_{\ell_k}^{\ceil*{\frac{i\la\la_{\ell_k}}{m}}}}{y^{i\la} \prod_{k=0}^{s} z_{\ell_k}^{j_k} }\right)_{\infty} = \sum_{k=0}^{s}(mj_k+i\la\la_{\ell_k})Q_{\ell_k}.$$
\end{proof}

For $s\geq 2$, the above proposition shows that $\Gamma_\lambda(Q_{\ell_0},\dots,Q_{\ell_s})\subseteq\Gamma(Q_{\ell_0},\dots, Q_{\ell_s})$. Next we show that $\Gamma(Q_{\ell_0},\dots,Q_{\ell_s})\subseteq\Gamma_\lambda(Q_{\ell_0},\dots, Q_{\ell_s})$.
\begin{proposition}\label{Gamma_sub_S}
	Let $1\leq s\leq \bar{r}$ and $\la\in \Z$ with $\gcd(m,\la) = 1$. Suppose that $d_{\ell_k} = \gcd(m,\la_{\ell_k}) = 1$ for all $0\leq k\leq s$.
	Then 
	$$\Gamma(Q_{\ell_0},\dots, Q_{\ell_s})\subseteq \Gamma_\lambda(Q_{\ell_0},\dots, Q_{\ell_s}).$$
\end{proposition}
\begin{proof}
    We will prove the proposition by induction on $s$.
	By Proposition \ref{two_places_Gamma}, we have $\Gamma_\la(Q_{\ell_0},Q_{\ell_1})=\Gamma(Q_{\ell_0},Q_{\ell_1})$,
	which settles the case where $ s=1 $. Assume that $\Gamma(Q_{\ell_0},\dots,Q_{\ell_{k}}) \subseteq \Gamma_\lambda(Q_{\ell_0},\dots,Q_{\ell_{k}})$ holds for all $ 1 \leq k \leq s-1 $, where $s\geq 2$. Suppose that $ \n=(n_0,\dots,n_s) \in \Gamma(Q_{\ell_0},\dots, Q_{\ell_s})$. Then there exists $ h \in F $ such that $ (h)_{\infty} = n_0Q_{\ell_0}+\dots+n_sQ_{\ell_s} $. 
	By Lemmas \ref{lemma5}, \ref{i_is_same} and Proposition \ref{gap_set}, we get 
	$$\mathbf{n} = (n_0,n_1,\dots,n_s) = (mj_0+i\lambda\la_{\ell_0},mj_1+i\lambda\la_{\ell_1},\dots,mj_{s}+i\lambda\la_{\ell_{s}}),$$
	where $ 1 \leq i \leq m-1 $ and $ \ceil*{\frac{-i\lambda\la_{\ell_k}}{m}}\leq j_k\leq \sum_{t=0}^{r}\ceil*{\frac{i \lambda\la_t}{m}}d_t-\ceil*{\frac{i \lambda\la_{\ell_k}}{m}}-1 $ for all $0\leq k \leq s$.
	
	Let $a_k = j_k+\ceil*{\frac{i\lambda\la_{\ell_k}}{m}}$ for $1\leq k\leq s$.
	Without loss of generality, we assume that $a_1 = \max\{a_1,\dots,a_s\}$.
	We have
	$$\left(h z_{\ell_s}^{j_s+\ceil*{\frac{i\lambda\la_{\ell_s}}{m}}}z_{\ell_1}^{-j_s-\ceil*{\frac{i\lambda\la_{\ell_s}}{m}}}\right)_{\infty}= n_0Q_{\ell_0} + \left(n_1+m\left(j_s+\ceil*{\frac{i\lambda\la_{\ell_s}}{m}}\right)\right)Q_{\ell_1} +\sum_{k=2}^{s-1}n_k Q_{\ell_k}.$$
	Thus 
	$\left(n_0,n_1+m\left(j_s+\ceil*{\frac{i\lambda\la_{\ell_s}}{m}}\right),n_2,\dots,n_{s-1}\right)\in H(Q_{\ell_0},\dots, Q_{\ell_{s-1}})$.
	By Theorem \ref{lub_semi_group}, there exists $ \mathbf{u}=(u_0,\dots,u_{s-1})\in \tilde{\Gamma}(Q_{\ell_0},\dots, Q_{\ell_{s-1}}) $ such that
	$$\mathbf{u}\preceq \left(n_0,n_1+m\left(j_s+\ceil*{\frac{i\lambda\la_{\ell_s}}{m}}\right),n_2,\dots,n_{s-1}\right),$$
	and $u_0 = n_0 = mj_0+i\lambda\la_{\ell_0}$.
	If $ u_1 \leq n_1 $, then $ (u_0,u_1,\dots,u_{s-1},0) \prec \mathbf{n} $. This yields a contradiction as $\mathbf{n}$ is minimal in  $\{\mathbf{w} = (w_0,\dots,w_{s})\in H(Q_{\ell_0},\dots, Q_{\ell_s})\mid w_0=mj_0+i\lambda\la_{\ell_0}\}$.
	Thus, we have $ u_1 > n_1 >0 $. Let $M = \{k_0,\dots, k_t\}= \{0\leq k\leq s-1 ~|~u_k>0\}$.
	We have $t\geq 1$ since $u_0=n_0>0$ and $u_1>n_1>0$.
	Then $\pi_M(\mathbf{u})\in \Gamma(Q_{\ell_{k_0}},\dots, Q_{\ell_{k_{t}}})$.
	By the induction hypothesis, we obtain that
	$$\pi_M(\mathbf{u}) = \left(mT_{k_0}+i\lambda\la_{\ell_{k_0}},\dots, m T_{k_t}+i\lambda\la_{\ell_{k_t}}\right)\in \Gamma_\lambda(Q_{\ell_{k_0}},\dots, Q_{\ell_{k_{t}}}),$$
	where $ 1 \leq i \leq m-1 $, $ T_{k_j}\geq \ceil*{\frac{-i\lambda\la_{\ell_{k_j}}}{m}} $ for $ 0\leq j \leq t$, and $ \sum_{j=0}^t T_{k_j} =\sum_{j=0}^r\ceil*{\frac{i\lambda\la_j}{m}}d_j-\sum_{j=0}^t \ceil*{\frac{i\lambda\la_{\ell_{k_j}}}{m}} $.
	Note that $k_0 = 0$ as $u_0 = n_0>0$ and $ k_1=1 $ as $ u_1 > n_1 >0 $.
	Since $mT_1+ i\lambda\la_{\ell_1} = u_1 > n_1  = mj_1+i\lambda\la_{\ell_1}$, then $T_1\geq j_1+1$.
	Since $j_1+\ceil*{\frac{i\lambda\la_{\ell_1}}{m}} = a_1\geq a_s = j_s+\ceil*{\frac{i\lambda\la_{\ell_s}}{m}} $, we get
	\begin{align*}
		T_1 - j_s - \ceil*{\frac{i\lambda\la_{\ell_s}}{m}}&\geq j_1+1-j_s- \ceil*{\frac{i\lambda\la_{\ell_s}}{m}}\geq 1-\ceil*{\frac{i\lambda\la_{\ell_1}}{m}}=\ceil*{\frac{-i\lambda\la_{\ell_1}}{m}}.
	\end{align*}
	Set
	$$\v = (v_0,\dots,v_s) = \left(u_0,u_1-m\left(j_s + \ceil*{\frac{i\lambda\la_{\ell_s}}{m}}\right),u_2,\dots,u_{s-1},mj_s+i\lambda\la_{\ell_s}\right).$$
	We have $\v\preceq \mathbf{n} $.
	Let $L = \{k_0,\dots,k_t,s\}$.
	Then $\pi_L(\v)$ is formed by some of the non-zero coordinates of $\v$. We verify that
	$$T_0+T_1-\left(j_s+\ceil*{\frac{i\lambda\la_{\ell_s}}{m}}\right)+ \sum_{j=2}^{t} T_{k_j}+j_s =\sum_{j=0}^r\ceil*{\frac{i\lambda\la_j}{m}}d_j-\sum_{j=0}^{t}\ceil*{\frac{i\lambda\la_{\ell_{k_j}}}{m}}-\ceil*{\frac{i\lambda\la_{\ell_s}}{m}},$$
	which implies that $\pi_L(\v)\in\Gamma_\lambda(Q_{\ell_{k_0}},\dots,Q_{\ell_{k_t}},Q_{\ell_s}) $.
	By Proposition \ref{S_subseteq_Gamma}, we have that $\Gamma_\lambda(Q_{\ell_{k_0}},\dots,Q_{\ell_{k_t}},Q_{\ell_s})  \subseteq \Gamma(Q_{\ell_{k_0}},\dots,Q_{\ell_{k_t}},Q_{\ell_s})$. It follows that
	$\v\in \tilde{\Gamma}(Q_{\ell_{0}},\dots, Q_{\ell_s}) \subseteq H(Q_{\ell_{0}},\dots, Q_{\ell_s})$.
	Notice that $\v\preceq \mathbf{n} $ and $ \mathbf{n} \in \Gamma(Q_{\ell_{0}},\dots, Q_{\ell_s})$.
	Therefore, we get that $\mathbf{n} = \v$ otherwise $ \mathbf{n} $ is not minimal in $ \{\mathbf{w} = (w_0,\dots,w_s)\in H(Q_{\ell_{0}},\dots, Q_{\ell_s})\mid w_0=mj_0+i\lambda\la_{\ell_0}\} $.
	Since $n_k>0$ for all $0\leq k\leq s$, we conclude that $L = \{0,1,\dots,s\}$ and then
	$\mathbf{n} = \v= \pi_L(\v)\in \Gamma_\lambda(Q_{\ell_{0}},\dots, Q_{\ell_s})$. Therefore $\Gamma(Q_{\ell_0},\dots, Q_{\ell_s})\subseteq \Gamma_\lambda(Q_{\ell_0},\dots, Q_{\ell_s})$.
\end{proof}

Based on the two propositions above, we complete the proof of Theorem \ref{thm_minimal_set} as follows.
\begin{proof}[Proof of Theorem \ref{thm_minimal_set}]
	According to Propositions \ref{S_subseteq_Gamma} and \ref{Gamma_sub_S}, it remains to show that $\Gamma_1(Q_{\ell_0},\dots, Q_{\ell_s})\neq \varnothing$ if and only if $1\leq s\leq \max_{1\leq i\leq m-1}\sum_{k=0}^r\ceil* {\frac{i\la_{k}}{m}}d_k-1$. 
	Let $\mathbf{u}_{\j,i} = (mj_0+i\la_{\ell_0},\dots, mj_s+i\la_{\ell_s})\in \Gamma_1(Q_{\ell_0},\dots, Q_{\ell_s})$ for some $1\leq i\leq m-1$.
    Then
    \begin{align*}
        0\leq \sum_{k=0}^{s}\left(j_k - \ceil*{\frac{-i\la_{\ell_k}}{m}} \right) &= \sum_{k=0}^r\ceil*{\frac{i\la_k}{m}}d_k + \sum_{k=0}^s\left(-\ceil*{\frac{i\lambda_{\ell_k}}{m}} + \floor*{\frac{i\lambda_{\ell_k}}{m}}\right)\\
        &=\sum_{k=0}^r\ceil*{\frac{i\la_k}{m}}d_k- s-1.
    \end{align*}
	This implies that  $s \leq \sum_{k = 0}^r \ceil*{\frac{i\la_k}{m}}d_k-1$ for some $1\leq i\leq m-1$. Thus $1\leq s\leq \max_{1\leq i\leq m-1}\sum_{k=0}^r\ceil* {\frac{i\la_{k}}{m}}d_k-1$.
	
	Conversely, if $1\leq s\leq \max_{1\leq i\leq m-1}\sum_{k=0}^r\ceil* {\frac{i\la_{k}}{m}}d_k-1$, then $s\leq \sum_{k=0}^r\ceil* {\frac{i\la_{k}}{m}}d_k-1$ for some $1\leq i\leq m-1$.
	Let $j_0 = \sum_{k=0}^r\ceil* {\frac{i\la_{k}}{m}}d_k-s-1 + \ceil*{\frac{-i\lambda_{\ell_0}}{m}}$ and $j_k = \ceil*{\frac{-i\lambda_{\ell_k}}{m}}$ for $1\leq k\leq s$. It is straightforward to verify that $(mj_0+i\la_{\ell_0},\dots, mj_s+i\la_{\ell_s})\in \Gamma_1(Q_{\ell_0},\dots, Q_{\ell_s})$.
	Therefore, we conclude that $\Gamma_1(Q_{\ell_0},\dots, Q_{\ell_s})\neq \varnothing$ if and only if $1\leq s\leq \max_{1\leq i\leq m-1}\sum_{k=0}^r\ceil* {\frac{i\la_{k}}{m}}d_k-1$.
\end{proof}

\section{Some examples}\label{section5}
In this section, we apply our results to present explicit examples of certain function fields, including generalized Kummer extensions with the same multiplicities, function fields of GGS curves, and function fields of subcovers of the BM curve.

\subsection{Generalized Kummer extensions with the same multiplicities}
Let $q$ be a prime power and let $m\geq 2$ be an integer with $\gcd(m,q) = 1$.
In this subsection, we consider the function field $F = \mathbb{F}_q(x,y)/\mathbb{F}_q(x)$ defined by the equation:
\begin{equation}\label{equation2}
	y^m = \prod_{i=1}^r(x-\alpha_i)^\lambda,
\end{equation}
where $\gcd(r\lambda,m) = 1$, and $\alpha_1,\dots,\alpha_r\in K$ are pairwise distinct.
The genus of $F$ is given by $g = (m-1)(r-1)/2$.
Let $Q_\infty$ denote the place in $\P_F$ corresponding to the pole of $x$. For each $1\leq i\leq r$, let $Q_i$ denote the place in $\P_F$ corresponding to the zero of $x-\alpha_i$.
Let $F' = F\bar{\mathbb{F}}_q$ denote the constant field extension of $F$, where $\bar{\mathbb{F}}_q$ is the algebraic closure of $\mathbb{F}_q$.
Then there is exactly one place $Q'_\infty\in \mathbf{P}_{F'}$ lying over $Q_\infty$. 
According to \cite[Theorem 3.6.3]{stichtenothAlgebraicFunctionFields2009}, we have $H(Q_\infty) = H(Q'_\infty)$.
Then it follows from \cite[Theorem 3.2]{castellanosOneTwoPointCodes2016} that $H(Q_\infty) = H(Q'_\infty) = \langle m,r\rangle$. Next, we establish further properties of $H(Q_\infty)$ and $H(Q_i)$ for $1\leq i\leq r$.

\begin{proposition}\label{same_gap}
    (i) $H(Q_\infty)$ is symmetric, $m_{H(Q_\infty)} = \min\{m,r\}$ and $F_{H(Q_\infty)} = mr - m -r$.

    (ii) For each $1\leq i\leq r$,
    $$H(Q_i) = \left\langle m, m\left(r-1-\floor*{\frac{ir}{m}}\right)+i: 1\leq i\leq m-1\right\rangle,$$
    $m_{H(Q_i)} = m - \floor*{\frac{m}{r}}$ and $F_{H(Q_i)} =  m\left(r-\floor*{\frac{r}{m}}-2\right)+ \max\{1,\floor*{\frac{m}{r}}\}$.
    Moreover, we have that $H(Q_i)$ is symmetric if and only if $r\equiv -1 \pmod m$.
\end{proposition}
\begin{proof}
(i) It follows from Corollary \ref{condition2} that $H(Q_\infty)$ is symmetric, and hence $ F_{H(Q_\infty)} = 2g-1 = mr - m -r$. Since $H(Q_\infty) = \langle m,r\rangle$, we have $m_{H(Q_\infty)} = \min\{m,r\}$.

(ii) Without loss of generality, it suffices to show these properties at $Q_1$. Let $\lambda'$ be the inverse of $\lambda$ modulo $m$. By Lemma \ref{sum_mod},
$$\sum_{k=1}^{r}\ceil*{\frac{i\lambda\lambda'}{m}} + \ceil*{\frac{-ir\lambda\lambda'}{m}} = \sum_{k=1}^{r}\ceil*{\frac{i}{m}} + \ceil*{\frac{-ir}{m}} = r-\floor*{\frac{ir}{m}}$$
for each $1\leq i\leq m-1$. Then by Corollary \ref{semigroup},
$$H(Q_1) = \left\langle m, m\left(r-1-\floor*{\frac{ir}{m}}\right)+i: 1\leq i\leq m-1\right\rangle.$$

Now let $r-1-\floor*{\frac{ir}{m}} = 0$. Then $(r-1)m\leq ir\leq (r-1)m + (m-1)$, which implies that
$$m-\floor*{\frac{m}{r}} = \ceil*{\frac{(r-1)m}{r}}\leq i\leq \floor*{\frac{rm-1}{r}} = m-1.$$
If $m<r$, then $\{1\leq i\leq m-1\mid r-1-\floor*{\frac{ir}{m}} = 0\} = \varnothing$. It follows from Corollary \ref{m_and_F} that $m_{H(Q_1)} = m = m - \floor*{\frac{m}{r}}$.
If $m>r$, then $m-\floor*{\frac{m}{r}}\in\{1\leq i\leq m-1\mid r-1-\floor*{\frac{ir}{m}} = 0\}$. It follows from Corollary \ref{m_and_F} that $m_{H(Q_1)} = m - \floor*{\frac{m}{r}}$.
Thus $m_{H(Q_1)} = m - \floor*{\frac{m}{r}}$.

Let $k = \max\left\{1\leq i\leq m-1\mid r-\floor*{\frac{ir}{m}}= r-\floor*{\frac{r}{m}}\right \}$.
If $r>m$, then $k=1$. If $r<m$, then $\floor*{\frac{kr}{m}} = 0$. We obtain that $\frac{kr}{m}<1$ and then $k = \floor*{\frac{m}{r}}$.
Thus $k = \max\{1,\floor*{\frac{m}{r}}\}$. By Corollary \ref{m_and_F}, we obtain the desired expression of $F_{H(Q_1)}$.

It follows from Corollary \ref{condition2} that $H(Q_1)$ is symmetric if and only if $-r\lambda \equiv \lambda \pmod m$. Therefore $H(Q_1)$ is symmetric if and only if $r\equiv -1 \pmod m$.
\end{proof}

We observe that the description of \cite[Theorem 9]{yangWeierstrassSemigroups2017} is imprecise, as it omits the required condition $r+1\leq q$. Moreover, the argument presented in the proof of \cite[Theorem 9]{yangWeierstrassSemigroups2017} (page 272, line 15), where the authors claim that “Without loss of generality, we may assume that $j_l = \max\{j_i\mid 0\leq i\leq l-1\}$”, is not fully rigorous.
The reason is that the expression of the set $\Gamma(P_\infty,P_1,\dots,P_l)$ at $P_\infty$ is not symmetric to those at $P_1,\dots,P_l$ in their proof, so the subsequent argument lacks essential considerations.
This in turn renders the proof of \cite[Theorem 10]{yangWeierstrassSemigroups2017} less rigorous, as \cite[Theorem 9]{yangWeierstrassSemigroups2017} is used therein.

In what follows, we present a revised statement of these results together with a more rigorous proof on a more general Kummer extension.
Compared with \cite[Theorems 9 and 10]{yangWeierstrassSemigroups2017}, we add the prerequisite condition $r+1\leq q$, and adopt a more precise range for $i$ in our expressions.

\begin{proposition} \label{Gamma_equal}
	Suppose that $r+1\leq q$. 
	
	(1) For $1\leq s\leq r-\floor*{\frac{r}{m}}-1$, the set $\Gamma(Q_\infty, Q_1,\dots,Q_s)$ is given by
	\begin{align*}
		\Big\{(mj_0 - ir, m&j_1+i,\dots,mj_s+i)\in \N^{s+1} ~\Big|~ 1\leq i\leq m-1-\floor*{\frac{ms}{r}},\\
		&j_0\geq \ceil*{\frac{ir}{m}},\, j_k\geq 0 \text{ for } 1\leq k\leq s \text{~and~} \sum_{k=0}^s j_k = r-s\Big\},
	\end{align*}
	and $\Gamma(Q_\infty,Q_1,\dots,Q_s) = \varnothing$ for $r-\floor*{\frac{r}{m}}-1 < s\leq r$.

	(2) For $2\leq s\leq r-\floor*{\frac{r}{m}}$, the set $\Gamma(Q_1,\dots,Q_s)$ is given by
	\begin{align*}
		\Big\{(mj_1+i,\dots,mj_s+i)\in \N^{s} ~&\Big|~ 1\leq i\leq m-\ceil*{\frac{m(s-1)}{r}},\\
		&j_k\geq 0 \text{ for } 1\leq k\leq s \text{~and~} \sum_{k=1}^s j_k = r-s - \floor*{\frac{ir}{m}}\Big\},
	\end{align*}
	and $\Gamma(Q_1,\dots,Q_s) = \varnothing$ for $r-\floor*{\frac{r}{m}} < s\leq r$.
\end{proposition}
\begin{proof}
    By Lemma \ref{sum_mod}, we have
    $$\max_{1\leq i\leq m-1}\left\{\ceil*{\frac{-ir\la}{m}} + r\ceil*{\frac{i\la}{m}}\right\} = \max_{1\leq i\leq m-1}\left\{\ceil*{\frac{-ir}{m}} + r\ceil*{\frac{i}{m}}\right\} = \ceil*{\frac{-r}{m}}+r = r-\floor*{\frac{r}{m}}.$$
	Let $\lambda'\in\Z$ be the inverse of $\lambda$ modulo $m$. Then there exists $b\in\Z$ such that $\lambda\lambda' = 1 + bm$.
	
	(1) By Corollary \ref{Gamma_some_equal}, for $1\leq s\leq r-\floor*{\frac{r}{m}}-1$, the set $\Gamma(Q_\infty,Q_1,\dots,Q_s)$ is given by
    \begin{align*}
		\Gamma(Q_\infty, Q_1,\dots,Q_s) = \Big\{(mj'_0-ir\lambda\lambda',mj_1+i,\dots,mj_s+i)\in \N^{s+1} ~\Big|~ 1\leq i\leq m-1,\\
		j_0'\geq \ceil*{\frac{ir\lambda\lambda'}{m}},\,j_k\geq 0 \text{ for } 1\leq k\leq s, \text{ and } \sum_{k=0}^s j_k =  r\ceil*{\frac{i\lambda\lambda'}{m}} - s\Big\},
	\end{align*}
    and $\Gamma(Q_\infty,Q_1,\dots,Q_s) = \varnothing$ for $r-\floor*{\frac{r}{m}}-1< s\leq r$. Note that 
    $$r\ceil*{\frac{i\lambda\lambda'}{m}} = r\ceil*{\frac{i(1+bm)}{m}}  = r+irb \text{ and } \ceil*{\frac{ir\lambda\lambda'}{m}} = \ceil*{\frac{ir(1+bm)'}{m}} = irb + \ceil*{\frac{ir}{m}}.$$
    Let $j_0 = j'_0 - irb$. Then we have $j_0\geq \ceil*{\frac{ir}{m}}$ and $\sum_{k=0}^s j_k = r-s$. 
	It follows from $\ceil*{\frac{ir}{m}}\leq j_0 \leq \sum_{k=0}^s j_k = r-s$ that $\frac{ir}{m}<r-s$, and then $i< m - \frac{ms}{r}$.
	Thus
      $$i\leq \floor*{m-\frac{ms}{r}} = m - \ceil*{\frac{ms}{r}} = m-1-\floor*{\frac{ms}{r}}.$$
	Therefore, we obtain the desired expression.

	(2) By Corollary \ref{Gamma_some_equal}, for $2\leq s\leq r-\floor*{\frac{r}{m}}$, the set $\Gamma(Q_1,\dots,Q_s)$ is given by
    \begin{align*}
		\Gamma(Q_1,\dots,Q_s) = \Big\{&(mj_1+i,\dots,mj_s+i)\in \N^{s} ~\Big|~ 1\leq i\leq m-1,\\
		&j_k\geq 0 \text{ for } 1\leq k\leq s, \text{ and } \sum_{k=1}^s j_k =  r\ceil*{\frac{i\lambda\lambda'}{m}} + \ceil*{\frac{-ir\lambda\lambda'}{m}}- s\Big\},
	\end{align*}
    and $\Gamma(Q_1,\dots,Q_s) = \varnothing$ for $r-\floor*{\frac{r}{m}}< s\leq r$. Note that 
    $$r\ceil*{\frac{i\lambda\lambda'}{m}} = r\ceil*{\frac{i(1+bm)}{m}}  = r+irb \text{ and } \ceil*{\frac{-ir\lambda\lambda'}{m}} = \ceil*{\frac{-ir(1+bm)'}{m}} =- irb - \floor*{\frac{ir}{m}}.$$
    Then we have $\sum_{k=1}^s j_k = r-s - \floor*{\frac{ir}{m}}$. 
	It follows from $0\leq \sum_{k=1}^s j_k = r-s - \floor*{\frac{ir}{m}}$ that $\floor*{\frac{ir}{m}}<r-s$, and then $\frac{ir}{m}<r-s+1$.
	Thus
      $$i\leq \floor*{\frac{m(r-s+1)}{r}} = m - \ceil*{\frac{(s-1)m}{r}}.$$
	Therefore, we obtain the desired expression.
    \end{proof}
\begin{example}
    Take $q = 25$, $m=6$, $r=5$ in the equation \eqref{equation2}. Consider the Kummer extension $F = \mathbb{F}_{25}(x,y)/\mathbb{F}_{25}(x)$ defined by the Hermitian curve $y^6 = x^5 + x$. Then $H(Q_\infty)$ is symmetric. Since $-5\equiv 1\pmod 6$, it follows from Proposition \ref{same_gap} that $H(Q_i)$ is also  symmetric for each $1\leq i\leq 5$. Moreover,
    $$m_{H(Q_\infty)} = m_{H(Q_i)} = 5 \text{ and } F_{H(Q_\infty)} = F_{H(Q_i)} = 19$$
    for each $1\leq i\leq 5$. By Proposition \ref{Gamma_equal}, we obtain that
    $$\Gamma(Q_\infty,Q_1) = \Gamma(Q_1,Q_2)= \left\{
	\begin{array}{l}
		(1,19),(2,14),(3,9),(4,4),(7,13),\\
        (8,8),(9,3),(13,7),(14,2),(19,1)
	\end{array}
\right\},$$

$$\Gamma(Q_\infty, Q_1,Q_2) = \Gamma(Q_1,Q_2,Q_3) =\left\{
	\begin{array}{l}
		(1,1,13),(2,2,8),(3,3,3),(1,13,1),(2,8,2),\\
		(13,1,1),(8,2,2),(1,7,7),(7,7,1),(7,1,7)
	\end{array}
\right\},$$

$$\Gamma(Q_\infty,Q_1,Q_2,Q_3) = \Gamma(Q_1,Q_2,Q_3,Q_4) = \left\{
	\begin{array}{ll}
		(1,1,1,7),(1,1,7,1),(1,7,1,1),\\
		(7,1,1,1),(2,2,2,2)
	\end{array}
\right\},$$

$$\Gamma(Q_\infty,Q_1,Q_2,Q_3,Q_4) = \Gamma(Q_1,Q_2,Q_3,Q_4,Q_5) = \left\{
	\begin{array}{l}
		(1,1,1,1,1)
	\end{array}
\right\},$$
and $\Gamma(Q_\infty,Q_1,Q_2,Q_3,Q_4,Q_5)  = \varnothing$.
\end{example}

\subsection{Function fields of GGS curves}
	Let $q$ be a prime power and let $n\geq 3$ be an odd integer. In this subsection, we consider the GGS curve $GGS(q,n)$, which is defined by the equations
	\begin{equation}
		\begin{cases}
		Y^{q+1} = X^q+X,\\
		Z^{\frac{q^n+1}{q+1}} = Y^{q^2}-Y.
		\end{cases}
	\end{equation}	
	The genus of $GGS(q,n)$ is $\frac{1}{2}(q-1)(q^{n+1}+q^n-q^2)$.
    The GGS curve is the first generalization of the GK curve \cite{garciaGeneralizationGiulietti2010}. When $n=3$, the curve $GGS(q,3)$ is a GK curve; see \cite{giuliettiNewFamily2009}.
	The curve is $\mathbb{F}_{q^{2n}}$-maximal. A plane model for the curve $GGS(q,n)$ can be given by 
    \begin{equation}\label{equation3}
	y^{q^n+1} = (x^q+x)((x^q+x)^{q-1}-1)^{q+1}.
	\end{equation}
    Note that $\mathbb{F}_{q^{2n}}(x,y)/\mathbb{F}_{q^{2n}}(x)$ is a Kummer extension. 
    We write 
    $$(x^q+x)((x^q+x)^{q-1}-1)^{q+1} = \left(\prod_{i=1}^q(x-\alpha_i)\right) \left(\prod_{i=1}^{q^2-q}(x-\beta_i)^{q+1}\right),$$
    where $\alpha_1,\dots,\alpha_q,\beta_1,\dots,\beta_{q^2-q}\in \mathbb{F}_{q^{2n}}$.
    Let $P_\infty$ be the pole of $x$ in $\mathbb{F}_{q^{2n}}(x)$. For each $1\leq i\leq q$, let $P_i$ be the zero of $x-\alpha_i$ in $\mathbb{F}_{q^{2n}}(x)$.
    It is obvious that $P_\infty$ and $P_i (1\leq i\leq q)$ are totally ramified in $\mathbb{F}_{q^{2n}}(x,y)/\mathbb{F}_{q^{2n}}(x)$. 
    Denote by $Q_\infty$ the only place lying over $P_\infty$, and by $Q_i$ the only place lying over $P_i$ for $1 \leq i \leq q$.

    \begin{proposition}\label{GGS_gap}
        Let $1\leq k\leq q$ and $m = (q^n+1)/(q+1)$. Then
        $$G(Q_k) = \left\{ (q^n+1)j + i \mid 1\leq i\leq q^n,\, 0\leq j\leq q + (q^2-q)\ceil*{\frac{i}{m}}-\floor*{\frac{iq^3}{q^n+1}}-2\right\},$$
        and
        $$H(Q_k) =  \left\langle q^n+1, (q^n+1)\left(q + (q^2-q)\ceil*{\frac{i}{m}}-\floor*{\frac{iq^3}{q^n+1}}-1\right)+i: 1\leq i\leq q^n\right\rangle.$$
        Moreover, $H(Q_k)$ is symmetric if and only if $n=3$.
    \end{proposition}
    \begin{proof}
        Note that 
        \begin{align*}
            &~q\ceil*{\frac{i}{q^n+1}} + (q^2-q)\ceil*{\frac{i(q+1)}{q^n+1}}+\floor*{\frac{-iq^3}{q^n+1}}\\
            =~&q + (q^2-q)\ceil*{\frac{i}{m}}-\floor*{\frac{iq^3}{q^n+1}}.
        \end{align*}
        Thus, by Proposition \ref{gap_set} and Corollary \ref{semigroup}, we have
          $$G(Q_k) = \left\{ (q^n+1)j + i \mid 1\leq i\leq q^n,\, 0\leq j\leq q + (q^2-q)\ceil*{\frac{i}{m}}-\floor*{\frac{iq^3}{q^n+1}}-2\right\},$$
          and
        $$H(Q_k) =  \left\langle q^n+1, (q^n+1)\left(q + (q^2-q)\ceil*{\frac{i}{m}}-\floor*{\frac{iq^3}{q^n+1}}-1\right)+i: 1\leq i\leq q^n\right\rangle.$$
        By Theorem \ref{symmetric_thm}, we get that $H(Q_k)$ is symmetric if and only if there exists $1\leq u\leq q^n$ such that 
		$$-uq^3\equiv 1\equiv u\pmod {q^n+1} \text{ and } u(q+1)\equiv q+1 \pmod {q^n+1}.$$
		Thus $H(Q_k)$ is symmetric if and only if $n=3$.
    \end{proof}

    \begin{corollary}
        Suppose that $n=3$. Then for each $1\leq s\leq q$, 
        \begin{align*}
            G(Q_\infty) = G(Q_s) = \{\,&j(q^3+1) + k(q^2-q+1)+t\mid\\
            & 0\leq k\leq q,\,1\leq t\leq q^2-q+1,\,0\leq j\leq q^2-1-t-k\}.
        \end{align*}
    \end{corollary}
    \begin{proof}
        Since $-q^3\equiv 1\pmod {q^3+1}$, it follows from Corollary \ref{semigroup_equal} that $G(Q_\infty) = G(Q_s)$ for $1\leq s\leq q$.
        Note that 
        $$\floor*{\frac{iq^3}{q^3+1}} = \floor*{\frac{iq^3+i-i}{q^3+1}} = i + \floor*{\frac{-i}{q^3+1}} = i-1$$
        for each $1\leq i\leq q^3$.
        By Proposition \ref{GGS_gap}, we obtain 
	\begin{align*}
		&G(Q_s)\\
        =~& \left\{ (q^3+1)j + i \mid 1\leq i\leq q^3,\, 0\leq j\leq q + (q^2-q)\ceil*{\frac{i}{q^2-q+1}} - \floor*{\frac{iq^3}{q^3+1}}-2\right\}\\
		=~& \left\{ (q^3+1)j + i \mid 1\leq i\leq q^3,\, 0\leq j\leq -i+q-1+(q^2-q)\ceil*{\frac{i}{q^2-q+1}}\right\}.
	\end{align*}
	Let $i = k(q^2-q+1)+t$, where $0\leq k\leq q$ and $1\leq t\leq q^2-q+1$. Then
	\begin{align*}
		-i+q-1+(q^2-q)\ceil*{\frac{i}{q^2-q+1}} &= -k(q^2-q+1)-t+q-1+(q^2-q)k + q^2-q\\
		&= q^2-1-t-k.
	\end{align*}
    The desired result follows.
    \end{proof}

\begin{proposition}\label{GGS_Gamma}
	Let $m = (q^n+1)/(q+1)$ and 
	$$2\leq s\leq \min\left\{q, \max_{1\leq i\leq q^n}\left\{ q + (q^2-q)\ceil*{\frac{i}{m}}-\floor*{\frac{iq^3}{q^n+1}}\right\}\right\}.$$
	Then
	\begin{align*}
		\Gamma(Q_1,\dots,Q_s) = \Big\{(&(q^n+1)j_1 + i,\dots,(q^n+1)j_s + i) \mid 1\leq i\leq q^n,\\
		&j_k\geq 0 \text{ for } 1\leq k\leq s,\,\sum_{k=1}^s j_k = q-s + (q^2-q)\ceil*{\frac{i}{m}}-\floor*{\frac{iq^3}{q^n+1}}\Big\}.
	\end{align*}
\end{proposition}
\begin{proof}
	Note that 
        \begin{align*}
            &~q\ceil*{\frac{i}{q^n+1}} + (q^2-q)\ceil*{\frac{i(q+1)}{q^n+1}}+\floor*{\frac{-iq^3}{q^n+1}}\\
            =~&q + (q^2-q)\ceil*{\frac{i}{m}}-\floor*{\frac{iq^3}{q^n+1}}.
        \end{align*}
	Setting $\lambda=1$ in Corollary \ref{Gamma_some_equal}, the desired result follows.
\end{proof}
\begin{example}
    Take $q = 2$, $n=3$ in the equation \eqref{equation3}. Consider the Kummer extension $F = \mathbb{F}_{2^6}(x,y)/\mathbb{F}_{2^6}(x)$ defined by $y^{9} = (x^2+x)(x^2+x-1)^3$. Then by Proposition \ref{GGS_gap}, we have $H(Q_\infty)$, $H(Q_1)$ and $H(Q_2)$ are symmetric. Moreover, we have 
    $$G(Q_\infty) = G(Q_1) = G(Q_2) = \{1,2,3,4,5,7,10,11,13,19\} \text{ and }$$
    $$H(Q_\infty) = H(Q_1) = H(Q_2) = \langle28,20,12,22,14,6,16,8,9\rangle.$$
    By Proposition \ref{GGS_Gamma}, we obtain
    $$\Gamma(Q_1,Q_2) = \left\{\begin{array}{l}
		(1,19),(19,1),(2,11),(11,2),(3,3),\\
        (4,13),(13,4),(5,5),(7,7),(10,10)
	\end{array}\right\}.$$
\end{example}

\subsection{Function fields of subcovers of the BM curve}
Let $q$ be a prime power and let $n\geq 3$ be an odd integer. Let $m$ be a divisor of $q^n+1$ and $d$ be a divisor of $q+1$ such that $\gcd(m,d(q-1)) = 1$.
In this subsection, we study the $\mathbb{F}_{q^{2n}}$-maximal curve $\mathcal{Y}_{d,m}$ defined by the affine equation 
	\begin{equation}\label{equation4}
		\mathcal{Y}_{d,m}:~~y^m = x^d(x^d-1)\left(\frac{1-x^{d(q-1)}}{x^d-1}\right)^{q+1}.
	\end{equation}
This curve was introduced in \cite[Theorem 3.1]{mendozaExplicitEquations2022}, and it is a subcover of the BM curve given by Beelen and Montanucci in \cite{beelenNewFamily2018}. Note that $\mathbb{F}_{q^{2n}}(x,y)/\mathbb{F}_{q^{2n}}(x)$ is a Kummer extension.
 We write 
    $$x^d(x^d-1)\left(\frac{1-x^{d(q-1)}}{x^d-1}\right)^{q+1} = x^d\left(\prod_{i=1}^d(x-\alpha_i)\right)\left(\prod_{i=1}^{d(q-2)}(x-\beta_i)^{q+1}\right) ,$$
    where $\alpha_1,\dots,\alpha_d,\beta_1,\dots,\beta_{d(q-2)}\in \mathbb{F}_{q^{2n}}$.
    For each $1\leq i\leq d$, let $P_i$ be the zero of $x-\alpha_i$ in $\mathbb{F}_{q^{2n}}(x)$.
    It is obvious that $P_i$ $(1\leq i\leq d)$ are totally ramified in $\mathbb{F}_{q^{2n}}(x,y)/\mathbb{F}_{q^{2n}}(x)$. Denote by $Q_i$ the only place lying over $P_i$ for each $1\leq i\leq d$.

\begin{proposition}\label{BM_GH}
    (i) Let $1\leq k\leq d$. Then $G(Q_k)$ is given by 
    \begin{align*}
    \Big\{ mj + i \mid &1\leq i\leq m-1, 0\leq j\leq d+\ceil*{\frac{di}{m}}+d(q-2)\ceil*{\frac{i(q+1)}{m}} - \floor*{\frac{idq(q-1)}{m}}-2\Big\},
    \end{align*}
    and $H(Q_k)$ is given by
     \begin{align*}
    \Big\langle m,m\left(d+\ceil*{\frac{di}{m}}+d(q-2)\ceil*{\frac{i(q+1)}{m}} - \floor*{\frac{idq(q-1)}{m}}-1\right) + i:1\leq i\leq m-1\Big\rangle.
    \end{align*}

	(ii) Let 
	$$2\leq s\leq \min\left\{d, \max_{1\leq i\leq m-1}\left\{ d+\ceil*{\frac{di}{m}}+d(q-2)\ceil*{\frac{i(q+1)}{m}} - \floor*{\frac{idq(q-1)}{m}}\right\}\right\}.$$
	Then
	\begin{align*}
		\Gamma(Q_1,\dots,Q_s) = \Big\{(&mj_1 + i,\dots,mj_s + i) \mid 1\leq i\leq m-1,\,j_k\geq 0 \text{ for } 1\leq k\leq s,\\
		&\sum_{k=1}^s j_k = d-s+\ceil*{\frac{di}{m}}+d(q-2)\ceil*{\frac{i(q+1)}{m}} - \floor*{\frac{idq(q-1)}{m}}\Big\}.
	\end{align*}
\end{proposition}
\begin{proof}
    Note that $d+d + d(q-2)(q+1) = dq(q-1)$. We obtain 
        \begin{align*}
            &~\ceil*{\frac{di}{m}} + d\ceil*{\frac{i}{m}}+d(q-2)\ceil*{\frac{i(q+1)}{m}} + \ceil*{\frac{-idq(q-1)}{m}}\\
            =~&d+\ceil*{\frac{di}{m}}+d(q-2)\ceil*{\frac{i(q+1)}{m}} - \floor*{\frac{idq(q-1)}{m}}.
        \end{align*}
    Then the first statement follows from Proposition \ref{gap_set} and Corollary \ref{semigroup}.
	The second statement follows by setting $\lambda=1$ in Corollary \ref{Gamma_some_equal}. 
\end{proof}

\begin{example}
    Take $q=3$, $n=3$, $d=4$, $m=7$ in the equation \eqref{equation4}.
    Consider the Kummer extension $F = \mathbb{F}_{3^6}(x,y)/\mathbb{F}_{3^6}(x)$ defined by $y^{7} = x^4(x^4-1)(x^4+1)^4$. Then by Proposition \ref{BM_GH}, for each $1\leq i\leq 4$, we have 
    $$G(Q_i) = \left\{\begin{array}{l}
		1, 2, 3, 4, 5, 6, 8, 9, 10, 11, 13, 15,\\
        16, 17, 18, 20, 22, 23, 25, 29, 30, 32, 37, 44
	\end{array}\right\}, \text{ and }$$
    $$H(Q_i) = \langle7,36,51,24,39,12,27\rangle.$$
    Moreover, we obtain that 
    $$\Gamma(Q_1,Q_2) = \left\{\begin{array}{l}
    (1,29),(8,22),(15,15),(22,8),(29,1),(2,44),(9,37),(16,30),\\
    (23,23),(30,16),(37,9),(44,2),(3,17),(10,10),(17,3),(4,32),\\
    (11,25),(18,18),(25,11),(32,4),(5,5),(6,20),(13,13),(20,6)
    \end{array}\right\},$$
    $$\Gamma(Q_1,Q_2,Q_3) = \left\{\begin{array}{l}
    (1,1,22),(1,8,15),(1,15,8),(1,22,1),(8,1,15),(8,8,8),\\
    (8,15,1),(15,1,8),(15,8,1),(22,1,1),(2,2,37),(2,9,30),\\
    (2,16,23),(2,23,16),(2,30,9),(2,37,2),(9,2,30),(9,9,23),\\
    (9,16,16),(9,23,9),(9,30,2),(16,2,23),(16,9,16),(16,16,9),\\
    (16,23,2),(23,2,16),(23,9,9),(23,16,2),(30,2,9),(30,9,2),\\
    (37,2,2),(3,3,10),(3,10,3),(10,3,3),(4,4,25),(4,11,18),\\
    (4,18,11),(4,25,4),(11,4,18),(11,11,11),(11,18,4),\\
    (18,4,11), (18,11,4),(25,4,4),(6,6,13),(6,13,6),(13,6,6)
    \end{array}\right\},$$
    and $\Gamma(Q_1,Q_2,Q_3,Q_4)$ is given by 
    $$\left\{\begin{array}{l}
    (1,1,1,15),(1,1,8,8),(1,1,15,1),(1,8,1,8),(1,8,8,1),(1,15,1,1),\\
    (8,1,1,8),(8,1,8,1),(8,8,1,1),(15,1,1,1),(2,2,2,30),(2,2,9,23),\\
    (2,2,16,16),(2,2,23,9),(2,2,30,2),(2,9,2,23),(2,9,9,16),(2,9,16,9),\\
    (2,9,23,2),(2,16,2,16),(2,16,9,9),(2,16,16,2),(2,23,2,9),(2,23,9,2),\\
    (2,30,2,2),(9,2,2,23),(9,2,9,16),(9,2,16,9),(9,2,23,2),(9,9,2,16),\\
    (9,9,9,9),(9,9,16,2),(9,16,2,9),(9,16,9,2),(9,23,2,2),(16,2,2,16),\\
    (16,2,9,9),(16,2,16,2),(16,9,2,9),(16,9,9,2),(16,16,2,2),(23,2,2,9),\\
    (23,2,9,2),(23,9,2,2),(30,2,2,2),(3,3,3,3),(4,4,4,18),(4,4,11,11),\\
    (4,4,18,4),(4,11,4,11),(4,11,11,4),(4,18,4,4),(11,4,4,11),(11,4,11,4),\\
    (11,11,4,4),(18,4,4,4),(6,6,6,6)
    \end{array}\right\}.$$
\end{example}
\section*{Acknowledgments}
This work is supported by the National Natural Science Foundation of China (No. 12441107),  Guangdong Basic and Applied Basic Research Foundation (No. \seqsplit{2025A1515011764}), and the National Key Research and Development Program of China (No.~\seqsplit{2025YFA1017100}). 

\bibliographystyle{elsarticle-num}
\bibliography{refs}

@book{stichtenothAlgebraicFunctionFields2009,
  title = {Algebraic {{Function Fields}} and {{Codes}}},
  author = {Stichtenoth, Henning},
  year = {2009},
  series = {Graduate {{Texts}} in {{Mathematics}}},
  number = {254},
  publisher = {Springer Berlin Heidelberg},
  doi = {10.1007/978-3-540-76878-4},
  isbn = {978-3-540-76877-7 978-3-540-76878-4},
  edition = {Second},
  address = {Berlin}
}

@book{arbarelloGeometryAlgebraic1985,
  title = {Geometry of {{Algebraic Curves}}},
  author = {Arbarello, E. and Cornalba, M. and Griffiths, P. A. and Harris, J.},
  year = 1985,
  series = {Grundlehren Der Mathematischen {{Wissenschaften}}},
  volume = {267},
  publisher = {Springer New York},
  address = {New York, NY},
  doi = {10.1007/978-1-4757-5323-3},
  isbn = {978-1-4419-2825-2 978-1-4757-5323-3}
}

@book{villasalvadorTopicsTheory2006,
  title = {Topics in the Theory of Algebraic Function Fields},
  author = {Villa Salvador, Gabriel Daniel},
  year = 2006,
  series = {Mathematics : Theory \& Applications},
  publisher = {Birkh\"auser},
  address = {Boston ; Berlin},
  doi = {10.1007/0-8176-4515-2},
  isbn = {978-0-8176-4480-2},
}

@article{carvalhoGoppaCodes2005,
  title = {On {{Goppa Codes}} and {{Weierstrass Gaps}} at {{Several Points}}},
  author = {Carvalho, C{\'i}cero and Torres, Fernando},
  year = 2005,
  journal = {Designs, Codes and Cryptography},
  volume = {35},
  number = {2},
  pages = {211--225},
  issn = {0925-1022, 1573-7586},
  doi = {10.1007/s10623-005-6403-4}
}

@article{duursmaImprovedTwoPoint2011,
  title = {Improved {{Two-Point Codes}} on {{Hermitian Curves}}},
  author = {Duursma, Iwan M. and Kirov, Radoslav},
  year = 2011,
  journal = {IEEE Transactions on Information Theory},
  volume = {57},
  number = {7},
  pages = {4469--4476},
  issn = {0018-9448, 1557-9654},
  doi = {10.1109/TIT.2011.2146410}
}

@article{matthewsWeierstrassPairsMinimum2001,
  title = {Weierstrass {{Pairs}} and {{Minimum Distance}} of {{Goppa Codes}}},
  author = {Matthews, Gretchen L.},
  year = 2001,
  journal = {Designs, Codes and Cryptography},
  volume = {22},
  number = {2},
  pages = {107--121},
  issn = {1573-7586},
  doi = {10.1023/A:1008311518095}
}

@article{matthewsCodesSuzuki2004,
  title = {Codes {{From}} the {{Suzuki Function Field}}},
  author = {Matthews, G.L.},
  year = 2004,
  journal = {IEEE Transactions on Information Theory},
  volume = {50},
  number = {12},
  pages = {3298--3302},
  issn = {0018-9448},
  doi = {10.1109/TIT.2004.838102}
}

@article{sepulvedaWeierstrassSemigroup2014,
  title = {Weierstrass Semigroup and Codes over the Curve $y^q+y = x^{q^r+1}$},
  author = {Sep{\'u}lveda, Alonso and Tizziotti, Guilherme},
  year = 2014,
  journal = {Advances in Mathematics of Communications},
  volume = {8},
  number = {1},
  pages = {67--72},
  issn = {1930-5346},
  doi = {10.3934/amc.2014.8.67}
}

@article{castellanosWeierstrassSemigroupsPure2024,
  title = {Weierstrass Semigroups, Pure Gaps and Codes on Function Fields},
  author = {Castellanos, Alonso S. and Mendoza, Erik A. R. and Quoos, Luciane},
  year = {2024},
  journal = {Designs, Codes and Cryptography},
  volume = {92},
  number = {5},
  pages = {1219--1242},
  issn = {0925-1022, 1573-7586},
  doi = {10.1007/s10623-023-01339-w}
}

@article{mendozaKummerExtensions2023,
  title = {On {{Kummer}} Extensions with One Place at Infinity},
  author = {Mendoza, Erik A.R.},
  year = 2023,
  journal = {Finite Fields and Their Applications},
  volume = {89},
  pages = {102209},
  issn = {10715797},
  doi = {10.1016/j.ffa.2023.102209}
}

@article{beelenNewFamily2018,
  title = {A New Family of Maximal Curves},
  shorttitle = {A New Family of Maximal Curves},
  author = {Beelen, Peter and Montanucci, Maria},
  year = 2018,
  journal = {Journal of the London Mathematical Society},
  volume = {98},
  number = {3},
  pages = {573--592},
  issn = {00246107},
  doi = {10.1112/jlms.12144}
}

@article{abdonWeierstrassPoints2019,
  title = {Weierstrass Points on {{Kummer}} Extensions},
  author = {Abd{\'o}n, Miriam and Borges, Herivelto and Quoos, Luciane},
  year = 2019,
  journal = {Advances in Geometry},
  volume = {19},
  number = {3},
  pages = {323--333},
  issn = {1615-7168, 1615-715X},
  doi = {10.1515/advgeom-2018-0021}
}

@article{bartoliWeierstrassSemigroups2021a,
  title = {Weierstrass Semigroups at Every Point of the {{Suzuki}} Curve},
  author = {Bartoli, Daniele and Montanucci, Maria and Zini, Giovanni},
  year = 2021,
  journal = {Acta Arithmetica},
  volume = {197},
  number = {1},
  pages = {1--20},
  issn = {0065-1036, 1730-6264},
  doi = {10.4064/aa181203-24-2}
}

@article{beelenWeierstrassSemigroups2021,
  title = {Weierstrass Semigroups on the {{Skabelund}} Maximal Curve},
  author = {Beelen, Peter and Landi, Leonardo and Montanucci, Maria},
  year = 2021,
  journal = {Finite Fields and Their Applications},
  volume = {72},
  pages = {101811},
  issn = {10715797},
  doi = {10.1016/j.ffa.2021.101811}
}

@article{castellanosOneTwoPointCodes2016,
  title = {One- and {{Two-Point Codes Over Kummer Extensions}}},
  author = {Castellanos, Alonso S. and Masuda, Ariane M. and Quoos, Luciane},
  year = 2016,
  journal = {IEEE Transactions on Information Theory},
  volume = {62},
  number = {9},
  pages = {4867--4872},
  issn = {0018-9448, 1557-9654},
  doi = {10.1109/TIT.2016.2583437}
}

@article{montanucciAGCodes2020,
  title = {{{AG}} Codes from the Second Generalization of the {{GK}} Maximal Curve},
  author = {Montanucci, Maria and Pallozzi Lavorante, Vincenzo},
  year = 2020,
  journal = {Discrete Mathematics},
  volume = {343},
  number = {5},
  pages = {111810},
  issn = {0012365X},
  doi = {10.1016/j.disc.2020.111810}
}

@article{maAutomorphismGroups2016a,
  title = {On Automorphism Groups of Cyclotomic Function Fields over Finite Fields},
  author = {Ma, Liming and Xing, Chaoping and Yeo, Sze Ling},
  year = 2016,
  journal = {Journal of Number Theory},
  volume = {169},
  pages = {406--419},
  issn = {0022314X},
  doi = {10.1016/j.jnt.2016.05.026}
}

@article{montanucciAutomorphismGroup2024,
  title = {On the Automorphism Group of a Family of Maximal Curves Not Covered by the {{Hermitian}} Curve},
  author = {Montanucci, Maria and Tizziotti, Guilherme and Zini, Giovanni},
  year = 2024,
  journal = {Finite Fields and Their Applications},
  volume = {99},
  pages = {102498},
  issn = {10715797},
  doi = {10.1016/j.ffa.2024.102498}
}

@article{beelenFamilyNonisomorphic2025,
  title = {A Family of Non-Isomorphic Maximal Function Fields},
  author = {Beelen, Peter and Montanucci, Maria and Niemann, Jonathan and Quoos, Luciane},
  year = 2025,
  journal = {Mathematische Zeitschrift},
  volume = {309},
  number = {2},
  pages = {19},
  issn = {0025-5874, 1432-1823},
  doi = {10.1007/s00209-024-03650-1}
}

@article{niemannNonisomorphicMaximal2025,
  title = {Non-Isomorphic Maximal Function Fields of Genus $q - 1$},
  author = {Niemann, Jonathan},
  year = 2025,
  journal = {Finite Fields and Their Applications},
  volume = {106},
  pages = {102618},
  issn = {10715797},
  doi = {10.1016/j.ffa.2025.102618}
}

@article{kirfelMinimumDistance1995a,
  title = {The Minimum Distance of Codes in an Array Coming from Telescopic Semigroups},
  author = {Kirfel, C. and Pellikaan, R.},
  year = 1995,
  journal = {IEEE Transactions on Information Theory},
  volume = {41},
  number = {6},
  pages = {1720--1732},
  issn = {00189448},
  doi = {10.1109/18.476245}
}

@article{guneriAutomorphismGroup2013,
  title = {The Automorphism Group of the Generalized {{Giulietti}}--{{Korchm\'aros}} Function Field},
  author = {G{\"u}neri, Cem and {\"O}zdemiry, Mehmet and Stichtenoth, Henning},
  year = 2013,
  journal = {Advances in Geometry},
  volume = {13},
  number = {2},
  pages = {369--380},
  issn = {1615-7168, 1615-715X},
  doi = {10.1515/advgeom-2012-0040}
}

@article{tafazolianCurveYn2019,
  title = {On the Curve ${{Y}}^n = {{X}}^\ell({{X}}^m+1)$ over Finite Fields},
  author = {Tafazolian, Saeed and Torres, Fernando},
  year = 2019,
  journal = {Advances in Geometry},
  volume = {19},
  number = {2},
  pages = {263--268},
  issn = {1615-7168, 1615-715X},
  doi = {10.1515/advgeom-2017-0041}
}

@article{huMultipointCodesKummer2018,
  title = {Multi-Point Codes over {{Kummer}} Extensions},
  author = {Hu, Chuangqiang and Yang, Shudi},
  year = 2018,
  journal = {Designs, Codes and Cryptography},
  volume = {86},
  number = {1},
  pages = {211--230},
  issn = {0925-1022, 1573-7586},
  doi = {10.1007/s10623-017-0335-7}
}

@incollection{matthewsWeierstrassSemigroup2004,
  title = {The {{Weierstrass Semigroup}} of an $m$-Tuple of {{Collinear Points}} on a {{Hermitian Curve}}},
  booktitle = {Finite {{Fields}} and {{Applications}}},
  author = {Matthews, Gretchen L.},
  editor = {Goos, Gerhard and Hartmanis, Juris and Van Leeuwen, Jan and Mullen, Gary L. and Poli, Alain and Stichtenoth, Henning},
  year = 2004,
  volume = {2948},
  pages = {12--24},
  publisher = {Springer Berlin Heidelberg},
  address = {Berlin, Heidelberg},
  doi = {10.1007/978-3-540-24633-6_2},
  isbn = {978-3-540-21324-6 978-3-540-24633-6}
}

@article{castellanosWeierstrassSemigroup2018,
  title = {On {{Weierstrass}} Semigroup at $m$ Points on Curves of the Form $f(y)=g(x)$},
  author = {Castellanos, A.S. and Tizziotti, G.},
  year = 2018,
  journal = {Journal of Pure and Applied Algebra},
  volume = {222},
  number = {7},
  pages = {1803--1809},
  issn = {00224049},
  doi = {10.1016/j.jpaa.2017.08.007}
}

@article{kimIndexWeierstrass1994,
  title = {On the Index of the {{Weierstrass}} Semigroup of a Pair of Points on a Curve},
  author = {Kim, Seon Jeong},
  year = 1994,
  journal = {Archiv der Mathematik},
  volume = {62},
  number = {1},
  pages = {73--82},
  issn = {0003-889X, 1420-8938},
  doi = {10.1007/BF01200442}
}

@article{hommaWeierstrassSemigroup1996,
  title = {The {{Weierstrass}} Semigroup of a Pair of Points on a Curve},
  author = {Homma, Masaaki},
  year = 1996,
  journal = {Archiv der Mathematik},
  volume = {67},
  number = {4},
  pages = {337--348},
  issn = {0003-889X, 1420-8938},
  doi = {10.1007/BF01197599}
}

@article{maharajCodeConstruction2004,
  title = {Code {{Construction}} on {{Fiber Products}} of {{Kummer Covers}}},
  author = {Maharaj, H.},
  year = 2004,
  journal = {IEEE Transactions on Information Theory},
  volume = {50},
  number = {9},
  pages = {2169--2173},
  issn = {0018-9448},
  doi = {10.1109/TIT.2004.833356}
}

@article{matthewsWeierstrassSemigroups2005,
  title = {Weierstrass {{Semigroups}} and {{Codes}} from a {{Quotient}} of the {{Hermitian Curve}}},
  author = {Matthews, Gretchen L.},
  year = 2005,
  journal = {Designs, Codes and Cryptography},
  volume = {37},
  number = {3},
  pages = {473--492},
  issn = {0925-1022, 1573-7586},
  doi = {10.1007/s10623-004-4038-5}
}

@incollection{matthewsWeierstrassSemigroups2009,
  title = {On {{Weierstrass Semigroups}} of {{Some Triples}} on {{Norm-Trace Curves}}},
  booktitle = {Coding and {{Cryptology}}},
  author = {Matthews, Gretchen L.},
  editor = {Chee, Yeow Meng and Li, Chao and Ling, San and Wang, Huaxiong and Xing, Chaoping},
  year = 2009,
  volume = {5557},
  pages = {146--156},
  publisher = {Springer Berlin Heidelberg},
  address = {Berlin, Heidelberg},
  doi = {10.1007/978-3-642-01877-0_13},
  isbn = {978-3-642-01813-8 978-3-642-01877-0}
}

@incollection{matthewsMinimalGenerating2010,
  title = {Minimal Generating Sets of {{Weierstrass}} Semigroups of Certain $m$-Tuples on the Norm-Trace Function Field},
  booktitle = {Contemporary {{Mathematics}}},
  author = {Matthews, Gretchen L. and Peachey, Justin D.},
  editor = {McGuire, Gary and Mullen, Gary L. and Panario, Daniel and Shparlinski, Igor E.},
  year = 2010,
  volume = {518},
  pages = {315--326},
  publisher = {American Mathematical Society},
  address = {Providence, Rhode Island},
  doi = {10.1090/conm/518/10214},
  isbn = {978-0-8218-4786-2 978-0-8218-8197-2}
}

@article{yangWeierstrassSemigroups2017,
  title = {Weierstrass Semigroups from {{Kummer}} Extensions},
  author = {Yang, Shudi and Hu, Chuangqiang},
  year = 2017,
  journal = {Finite Fields and Their Applications},
  volume = {45},
  pages = {264--284},
  issn = {10715797},
  doi = {10.1016/j.ffa.2016.12.005}
}

@article{tizziottiWeierstrassSemigroup2018,
  title = {Weierstrass {{Semigroup}} and {{Pure Gaps}} at {{Several Points}} on the {{GK Curve}}},
  author = {Tizziotti, G. and Castellanos, A. S.},
  year = 2018,
  journal = {Bulletin of the Brazilian Mathematical Society, New Series},
  volume = {49},
  number = {2},
  pages = {419--429},
  issn = {1678-7544, 1678-7714},
  doi = {10.1007/s00574-017-0059-3}
}

@article{castellanosWeierstrassSemigroup2020,
  title = {Weierstrass Semigroup at $m+1$ Rational Points in Maximal Curves Which Cannot Be Covered by the {{Hermitian}} Curve},
  author = {Castellanos, Alonso Sep{\'u}lveda and {Bras-Amor{\'o}s}, Maria},
  year = 2020,
  journal = {Designs, Codes and Cryptography},
  volume = {88},
  number = {8},
  pages = {1595--1616},
  issn = {0925-1022, 1573-7586},
  doi = {10.1007/s10623-020-00757-4}
}

@article{mendozaExplicitEquations2022,
  title = {Explicit Equations for Maximal Curves as Subcovers of the {{BM}} Curve},
  author = {Mendoza, Erik A.R. and Quoos, Luciane},
  year = 2022,
  journal = {Finite Fields and Their Applications},
  volume = {77},
  pages = {101945},
  issn = {10715797},
  doi = {10.1016/j.ffa.2021.101945}
}

@article{castellanosGeneralizedWeierstrass2026,
  title = {On Generalized {{Weierstrass}} Semigroups in Arbitrary {{Kummer}} Extensions of $\mathbb{F}_q (x)$},
  author = {Castellanos, Alonso S. and Mendoza, Erik and Tizziotti, Guilherme},
  year = 2026,
  journal = {Finite Fields and Their Applications},
  volume = {112},
  pages = {102808},
  issn = {10715797},
  doi = {10.1016/j.ffa.2026.102808}
}

@article{garciaGeneralizationGiulietti2010,
title = {A generalization of the {{Giulietti}}--{{Korchm\'aros}} maximal curve},
author = {Arnaldo Garcia and Cem Güneri and Henning Stichtenoth},
pages = {427--434},
volume = {10},
number = {3},
journal = {Advances in Geometry},
doi = {10.1515/advgeom.2010.020},
year = {2010},
}

@article{giuliettiNewFamily2009,
  title = {A New Family of Maximal Curves over a Finite Field},
  author = {Giulietti, Massimo and Korchm{\'a}ros, G{\'a}bor},
  year = 2009,
  journal = {Mathematische Annalen},
  volume = {343},
  number = {1},
  pages = {229--245},
  issn = {0025-5831, 1432-1807},
  doi = {10.1007/s00208-008-0270-z}
}

@article{yangWeierstrassSemigroups2024a,
  title = {Weierstrass Semigroups on the Third Function Field in a Tower Attaining the {{Drinfeld-Vl\u{a}du\c t}} Bound},
  author = {Yang, Shudi and Hu, Chuangqiang},
  year = 2024,
  journal = {Advances in Mathematics of Communications},
  volume = {18},
  number = {4},
  pages = {1051--1083},
  issn = {1930-5346, 1930-5338},
  doi = {10.3934/amc.2022066}
}

@article{huMultipointCodes2020,
  title = {Multi-Point Codes from the {{GGS}} Curves},
  author = {Hu, Chuangqiang and Yang, Shudi},
  year = 2020,
  journal = {Advances in Mathematics of Communications},
  volume = {14},
  number = {2},
  pages = {279--299},
  issn = {1930-5338},
  doi = {10.3934/amc.2020020}
}

@article{beelenWeierstrassSemigroups2023,
  title = {Weierstrass Semigroups and Automorphism Group of a Maximal Curve with the Third Largest Genus},
  author = {Beelen, Peter and Montanucci, Maria and Vicino, Lara},
  year = 2023,
  journal = {Finite Fields and Their Applications},
  volume = {92},
  pages = {102300},
  issn = {10715797},
  doi = {10.1016/j.ffa.2023.102300}
}

@article{beelenWeierstrassSemigroups2026a,
  title = {Weierstrass Semigroups and Automorphism Group of a Maximal Function Field with the Third Largest Possible Genus, $q \equiv 1 \pmod 3$},
  author = {Beelen, Peter and Montanucci, Maria and Vicino, Lara},
  year = 2026,
  journal = {Finite Fields and Their Applications},
  volume = {109},
  pages = {102701},
  issn = {10715797},
  doi = {10.1016/j.ffa.2025.102701}
}

@article{beelenWeierstrassSemigroups2026,
  title = {Weierstrass Semigroups and Automorphism Group of a Maximal Function Field with the Third Largest Possible Genus, $q \equiv 0 \pmod 3$},
  author = {Beelen, Peter and Montanucci, Maria and Vicino, Lara},
  year = 2026,
  journal = {Finite Fields and Their Applications},
  volume = {110},
  pages = {102729},
  issn = {10715797},
  doi = {10.1016/j.ffa.2025.102729}
}

@misc{cotterillGapSets2025,
  title = {On Gap Sets in Arbitrary {{Kummer}} Extensions of ${{K}}(x)$},
  author = {Cotterill, Ethan and Mendoza, Erik A. R. and Speziali, Pietro},
  year = {2025},
  number = {arXiv:2506.19169},
  eprint = {2506.19169},
  primaryclass = {math},
  publisher = {arXiv},
  doi = {10.48550/arXiv.2506.19169},
  archiveprefix = {arXiv},
}
    
\end{document}